\documentclass[10pt]{article}
\usepackage{amsthm}
\usepackage{amsmath}
\usepackage{amssymb}
\usepackage{makeidx}

\theoremstyle{definition}
\newtheorem{definition}{Definition}[section]
\newtheorem{notation}[definition]{Notation}
\newtheorem{example}[definition]{Example}
\newtheorem{openproblem}[definition]{Open problem}

\theoremstyle{plain}
\newtheorem{theorem}[definition]{Theorem}
\newtheorem{lemma}[definition]{Lemma}
\newtheorem{proposition}[definition]{Proposition}
\newtheorem{corollary}[definition]{Corollary}
\newtheorem{remark}[definition]{Remark}

\newcommand{\beq}{\begin{equation}}

\newcommand{\eeq}{\end{equation}}

\newcommand{\bdfn}{\begin{definition}}

\newcommand{\edfn}{\end{definition}}

\newcommand{\bthm}{\begin{theorem}}

\newcommand{\ethm}{\end{theorem}}

\newcommand{\bprop}{\begin{proposition}}

\newcommand{\eprop}{\end{proposition}}

\newcommand{\bcor}{\begin{corollary}}

\newcommand{\ecor}{\end{corollary}}

\newcommand{\blem}{\begin{lemma}}

\newcommand{\elem}{\end{lemma}}

\newcommand{\bex}{\begin{example}}

\newcommand{\eex}{\end{example}}

\newcommand{\bxc}{\begin{exercise}}

\newcommand{\exc}{\end{exercise}}

\newcommand{\bntn}{\begin{notation}}

\newcommand{\entn}{\end{notation}}

\newcommand{\be}{\begin{enumerate}}

\newcommand{\ee}{\end{enumerate}}

\newcommand{\bce}{\begin{center}}

\newcommand{\ece}{\end{center}}

\newcommand{\bi}{\begin{itemize}}

\newcommand{\ei}{\end{itemize}}

\newcommand{\bt}{\begin{tabular}}

\newcommand{\et}{\end{tabular}}

\newcommand{\si}{\wedge}

\newcommand{\sau}{\vee}

\newcommand{\ba}{\begin{array}} 

\newcommand{\ea}{\end{array}}

\numberwithin{equation}{section}

\def\R{{\mathbb R}}

\def\N{{\mathbb N}}

\newcommand {\bua} {\begin{eqnarray*}}

\newcommand {\eua} {\end {eqnarray*}}

\begin{document}
\title{Generalized Bosbach States}
\author{George GEORGESCU and Claudia MURE\c{S}AN\\ \footnotesize University of Bucharest\\ \footnotesize Faculty of Mathematics and Computer Science\\ \footnotesize Academiei 14, RO 010014, Bucharest, Romania\\ \footnotesize Emails: georgescu@funinf.cs.unibuc.ro, c.muresan@yahoo.com}
\date{}
\maketitle

\begin{abstract}

Bosbach states represent a way of probabilisticly evaluating the formulas from various (commutative or non-commutative) many-valued logics. They are defined on the algebras corresponding to these logics with values in $[0,1]$. Starting from the observation that in the definition of Bosbach states there intervenes the standard MV-algebra structure of $[0,1]$, in this paper we introduce Bosbach states defined on residuated lattices with values in residuated lattices. We are led to two types of generalized Bosbach states, with distinct behaviours. The properties of generalized Bosbach states, proven in the paper, may serve as an algebraic foundation for developping some probabilistic many-valued logics.


\end{abstract}

\section{Introduction}

\hspace*{10pt} Classical probability theory is based on the hypothesis that the sets of events associated with random experiments have a structure of a Boolean algebra. This fact derives from the thesis that the random experiment follows the rules of classical logic. An important part of probability theory can be developped by considering probabilities on arbitrary Boolean algebras (\cite{cuon}, \cite{frem}) .

It can happen for random experiments to follow the rules of another logical system. Then the sets of events will have the structure of the Lindenbaum-Tarski algebra associated to that logical system.

In the case of infinite-valued \L ukasiewicz logic, the sets of events will have a structure of MV-algebra (\cite{cdom}). The study of probabilities defined on MV-algebras (which are called {\em MV-states}) has been started in \cite{mun} and then continued by numerous authors (see, for instance, \cite{br1}, \cite{br3}, \cite{ioanal2}).

Together with these, there have been studied different types of states defined on pseudo-MV-algebras (\cite{dvus}), BL-algebras (\cite{br2}), pseudo-BL-algebras (\cite{gg}), Rl-monoids (\cite{dvura}, \cite{dvura2}), residuated lattices (\cite{lcciu}, \cite{lcciu1}), pseudo-BCK-algebras (\cite{kuhr}) etc..

Bosbach states, introduced in \cite{gg}, have as domain a pseudo-BL-algebra $A$ and as codomain the real interval $[0,1]$. The axioms of the Bosbach states are expressed in terms of the two implications of $A$ and of the addition in $\R $.

But states can be thought of in another way. By identifying an event with the sentence that describes that event, states will become functions defined on the set of the sentences of the logical system and having as target set the real interval $[0,1]$. This way states can be regarded as a type of semantics. This point of view suggests us to consider $[0,1]$ as a standard algebra of a logical system and to report the definition of states to this algebra.

The present work starts from the observation that Bosbach states can be defined using the canonical structure of standard MV-algebra of $[0,1]$.

By replacing the MV-algebra $[0,1]$ with an arbitrary residuated lattice $L$, we aim to find a concept of a state (called {\em generalized Bosbach state}) defined on an arbitrary residuated lattice $A$ and with $L$ as target set. To this end, we will express the definition of the Bosbach state in several equivalent forms. By comparing these equivalent forms we will obtain two notions of generalized Bosbach states: of type I and of type II.

We will notice that type I states are not order-preserving. By considering order preservation as an essential property for any notion of state, we will be studying especially order-preserving type I states. We will study in parallel order-preserving type I states and type II states. By analyzing the way in which some properties of Bosbach states can be extended to type I and type II states, we will notice a strong asymmetry between them.

This paper is organized as follows. In Section \ref{preliminaries} we present some basic definitions and results from the theory of residuated lattices. Section \ref{bosbach} contains the definition of generalized Bosbach states (of type I and of type II), preceeded by a detailed discussion on its motivation. We give several examples and we prove some arithmetic properties of generalized Bosbach states, as well as some characterizations of them. Section \ref{filtcan} deals with the properties of the canonical filter associated with a generalized Bosbach state and of the corresponding quotient residuated lattice. These are related to the notion of state-morphism, which generalizes the one from \cite{dvus}, \cite{dvura}, \cite{gg} to the more general context of this paper. In Section \ref{riecan} we introduce generalized Rie\v can states. These extend the concept of Rie\v can state from \cite{br2}, \cite{gg}, \cite{dvura}, \cite{lcciu}. We analyze the link between generalized Rie\v can states and generalized Bosbach states of type I and II. In Section \ref{convergente} we are treating the continuity of generalized Bosbach states. In \cite{ggap} the authors introduced the similarity convergence in the context of residuated lattices. Based on this similarity convergence, we are defining three types of continuity for generalized Bosbach states and we establish links between them. To each order-preserving type I state we can associate canonically a similarity relation, which allows us to accomplish, in the general case of the present paper, a construction that generalizes the metric completion of an MV-algebra. The last section of this paper contains a sketch of some connections between generalized Bosbach states and some many-valued logical systems.
 
\section{Preliminaries}
\label{preliminaries}

\hspace*{10pt} In this section we recall some notions and arithmetic properties of several varieties of residuated lattices and some from the theory of filters and congruences of residuated lattices. We refer the reader to \cite{beloh}, \cite{cdom}, \cite{haj}, \cite{ior}, \cite{kow}.

\begin{definition}
A {\em residuated lattice} is an algebraic structure of the form $(A,\vee ,\wedge ,\odot ,\rightarrow ,0,1)$, in which: $(A,\vee ,\wedge ,0,1)$ is a bounded lattice, $(A,\odot ,1)$ is a commutative monoid and, for all $a,b,c\in A$, $a\leq b\rightarrow c$ iff $a\odot b\leq c$ ({\em the law of residuation}).

\end{definition}

For any residuated lattice $A$ and any $a,b\in A$, we denote $\neg \, a=a\rightarrow 0$ ({\em the negation}) and $a\leftrightarrow b=(a\rightarrow b)\wedge (b\rightarrow a)$ ({\em the biresiduum} or {\em the equivalence}). We will also denote $d_A(a,b)=a\leftrightarrow b$.

The next two lemmas collect several arithmetic properties of residuated lattices.

\begin{lemma}{\rm \cite{kow}, \cite{haj}, \cite{pic}, \cite{tur}} For any residuated lattice $A$ and any $a,b,c,d\in A$, we have:
\begin{enumerate}
\item\label{l2.2(1)} $a\rightarrow 1=1$;
\item\label{l2.2(2)} $1\rightarrow a=a$;
\item\label{l2.2(3)} $a\leq b$ iff $a\rightarrow b=1$;
\item\label{l2.2(4)} if $a\leq b$ then $b\rightarrow c\leq a\rightarrow c$ and $c\rightarrow a\leq c\rightarrow b$;
\item\label{l2.2(5)} $a\odot b\leq a\wedge b\leq d_A(a,b)\leq a\rightarrow b$;
\item\label{l2.2nenum} $b\leq a\rightarrow b$;
\item\label{l2.2(6)} $a\odot (a\rightarrow b)\leq b$;
\item\label{l2.2nenum2} if $a\leq c$ and $b\leq d$, then $a\odot b\leq c\odot d$;
\item\label{l2.2(7)} $a\rightarrow (b\rightarrow c)=(a\odot b)\rightarrow c=b\rightarrow (a\rightarrow c)$;
\item\label{l2.2(8)} $(a\vee b)\odot c=(a\odot c)\vee (b\odot c)$;
\item\label{l2.2(9)} $(a\vee b)\rightarrow c=(a\rightarrow c)\wedge (b\rightarrow c)$ and $c\rightarrow (a\wedge b)=(c\rightarrow a)\wedge (c\rightarrow b)$; moreover, for any nonempty set $I$ and any family $(a_i)_{i\in I}\subseteq A$ such that $\bigvee _{i\in I}a_i$ exists, $(\bigvee _{i\in I}a_i)\rightarrow c=\bigwedge _{i\in I}(a_i\rightarrow c)$. 
\end{enumerate}
\label{l2.2}
\end{lemma}

\begin{lemma}{\rm \cite{beloh}, \cite{tur}} In a residuated lattice, the biresiduum has the following properties, for all $a,b,c,x,y\in A$:

\begin{enumerate}
\item\label{lA(1)} $d_A(a,b)=1$ iff $a=b$;
\item\label{lA(2)} $d_A(a,b)=d_A(b,a)$;
\item\label{lA(3)} $d_A(a,b)\odot d_A(b,c)\leq d_A(a,c)$;
\item\label{lA(4)} $d_A(a,b)\leq d_A(\neg \, a,\neg \, b)$;
\item\label{lA(5)} $d_A(a,b)\odot d_A(x,y)\leq d_A(a\circ x,b\circ y)$ for each $\circ \in \{\vee ,\wedge ,\odot ,\rightarrow ,\leftrightarrow \}$.
\end{enumerate}
\label{lA}
\end{lemma}

\begin{lemma}{\rm \cite{kow}, \cite{haj}, \cite{pic}, \cite{tur}} For any residuated lattice $A$ and any $a,b\in A$, we have:
\begin{enumerate}
\item\label{0neg1} $\neg \, 0=1$ and $\neg \, 1=0$;
\item\label{l2.3(1)} $a\leq \neg \, b$ iff $a\odot b=0$;
\item\label{l2.3(2)} $a\leq \neg \, \neg \, a$ and $\neg \, \neg \, \neg \, a=a$;
\item\label{l2.3(3)} if $a\leq b$ then $\neg \, b\leq \neg \, a$;
\item\label{l2.3(4)} $a\rightarrow b\leq \neg \, b\rightarrow \neg \, a$;
\item\label{l2.3(5)} $\neg \, (a\odot b)=a\rightarrow \neg \, b=b\rightarrow \neg \, a$.
\end{enumerate}
\label{l2.3}
\end{lemma}

Important classes of residuated lattices can be introduced starting from the notion of t-norm. A {\em t-norm} is a binary operation $\odot $ on $[0,1]$ with the properties of being associative, commutative, order-preserving and with $1$ as identity. If a t-norm $\odot $ is left-continuous, then we can consider the operation residuum $\rightarrow $ on $[0,1]$, defined by $a\rightarrow b=\max \{c\in [0,1]|c\odot a\leq b\}$. Then $([0,1],\max ,\min ,\odot ,\rightarrow ,0,1)$ is a residuated lattice.

A residuated lattice $A$ is called an {\em MTL-algebra} iff, for all $a,b\in A$, $(a\rightarrow b)\vee (b\rightarrow a)=1$. If $\odot $ is a left-continuous t-norm, then $([0,1],\max ,\min ,\odot ,\rightarrow ,0,1)$ is an MTL-algebra.

\begin{lemma}{\rm \cite{ego}} If $A$ is an MTL-algebra and $a,b\in A$, then $a\vee b=((a\rightarrow b)\rightarrow b)\wedge ((b\rightarrow a)\rightarrow a)$.
\label{lMTL}
\end{lemma}

A {\em BL-algebra} is an MTL-algebra $A$ with the property that, for all $a,b\in A$, $a\wedge b=a\odot (a\rightarrow b)$. If $\odot $ is a continuous t-norm, then $([0,1],\max ,\min ,\odot ,\rightarrow ,0,1)$ is a BL-algebra.

We list below the three fundamental continuous t-norms and their residua:

\begin{itemize}
\item {\em the \L ukasiewicz t-norm}: $a\odot _{L}b=\max \{0,a+b-1\}$, $a\rightarrow _{L}b=\min \{1,1-a+b\}$;
\item {\em the $G\ddot{o}del$ t-norm}: $a\odot _{G}b=\min \{a,b\}$, $a\rightarrow _{G}b=\begin{cases}1, & {\rm if}\ a\leq b,\\ b, & {\rm otherwise;}\end{cases}$
\item {\em the product} or {\em Gaines t-norm}: $a\odot _{P}b=a\cdot b$, $a\rightarrow _{P}b=\begin{cases}1, & {\rm if}\ a\leq b,\\ b/a, & {\rm otherwise.}\end{cases}$
\end{itemize}

An {\em MV-algebra} is an algebra $(A,\oplus,\neg \, ,0)$ with one binary operation $\oplus$, one unary operation $\neg \, $ and one constant 0 such that: $(A,\oplus,0)$ is a commutative monoid and, for all $a,b\in A$, $\neg \, \neg \, a=a$, $a\oplus \neg \, 0=\neg \, 0$, $\neg \, (\neg \, a\oplus b)\oplus b=\neg \, (\neg \, b\oplus a)\oplus a$. If $A$ is an MV-algebra, then the binary operations $\odot$, $\si$, $\sau$, $\rightarrow $ and the constant 1 are defined by the following relations: for all $a,b\in A$, $a\odot b=\neg \, (\neg \, a\oplus \neg \, b)$, $a\si b=(a\oplus \neg \, b)\odot b$ , $a\sau b=(a\odot \neg \, b)\oplus b$, $a\rightarrow b=\neg\, a\oplus b$, $1=\neg \, 0$. According to \cite[Theorem 3.2, page 99]{pic}, MV-algebras are exactly the involutive BL-algebras, that is: an MV-algebra is a BL-algebra $A$ with the property that, for all $a\in A$, $\neg \, \neg \, a=a$. $([0,1],\max ,\min ,\odot _{L},\rightarrow _{L},0,1)$ is an MV-algebra, called {\em the standard MV-algebra}.

\begin{lemma}{\rm \cite{ioanal}} Let $A$ be an MV-algebra and $a,b,c\in A$. Then:

\begin{enumerate}

\item\label{mvsum} $a\oplus \neg \, a=1$;
\item\label{mvdemorgan1} $\neg \, (a\odot b)=\neg \, a\oplus \neg \, b$;
\item\label{mvdemorgan2} $\neg \, (a\oplus b)=\neg \, a\odot \neg \, b$;
\item\label{mvdistrib} $a\oplus (b\wedge c)=(a\oplus b)\wedge (a\oplus c)$;
\item\label{mvvee} $c\rightarrow (a\vee b)=(c\rightarrow a)\vee (c\rightarrow b)$; moreover, for any nonempty set $I$ and any family $(a_i)_{i\in I}\subseteq A$ such that $\bigvee _{i\in I}a_i$ exists, $c\rightarrow (\bigvee _{i\in I}a_i)=\bigvee _{i\in I}(c\rightarrow a_i)$;
\item\label{mvwedge} for any nonempty set $I$ and any family $(a_i)_{i\in I}\subseteq A$ such that $\bigwedge _{i\in I}a_i$ exists, $(\bigwedge _{i\in I}a_i)\rightarrow c=\bigvee _{i\in I}(a_i\rightarrow c)$.

\end{enumerate}
\label{mvlema}
\end{lemma}

A {\em Heyting algebra} is a residuated lattice $A$ such that, for all $a,b\in A$, $a\odot b=a\wedge b$. In a Heyting algebra $A$ we have: $a\wedge (a\rightarrow b)=a\wedge b$ for all $a,b\in A$ (for a proof see, for instance, \cite[Proposition 1.20, page 17]{pic}).

A {\em $G\ddot{o}del$ algebra} is a BL-algebra $A$ such that, for all $a,b\in A$, $a\odot b=a\wedge b$, that is: both a Heyting algebra and a BL-algebra. $([0,1],\max ,\min ,\odot _{G},\rightarrow _{G},0,1)$ is a ${\rm G\ddot{o}del}$ algebra.

A {\em product} or {\em PL-algebra} is a BL-algebra $A$ that satisfies the following two conditions:

\begin{itemize}
\item for all $a\in A$, $a\wedge \neg \, a=0$;
\item for all $a,b,c\in A$, $(\neg \, \neg \, c\odot ((a\odot c)\rightarrow (b\odot c)))\rightarrow (a\rightarrow b)=1$.
\end{itemize}

$([0,1],\max ,\min ,\odot _{P},\rightarrow _{P},0,1)$ is a product algebra.

A residuated lattice $A$ is said to be {\em involutive} iff $\neg \, \neg \, a=a$ for all $a\in A$. A residuated lattice $A$ is said to be {\em divisible} iff $a\wedge b=a\odot (a\rightarrow b)$ for all $a,b\in A$. A divisible and involutive residuated lattice is an MV-algebra.

\begin{lemma}{\rm \cite{cdom}, \cite{haj}, \cite{ior}} A residuated lattice $A$ is an MV-algebra iff, for all $a,b\in A$, $(a\rightarrow b)\rightarrow b=(b\rightarrow a)\rightarrow a$. In this case, for all $a,b\in A$, $a\vee b=(a\rightarrow b)\rightarrow b=(b\rightarrow a)\rightarrow a$.
\label{l2.4}
\end{lemma}

Throughout the remaining part of this section, let $A$ be a residuated lattice. A {\em filter} of $A$ is a nonempty subset $F$ of $A$ such that, for all $a,b\in A$:

\begin{itemize}
\item $a,b\in F$ implies $a\odot b\in F$;
\item $a\in F$ and $a\leq b$ imply $b\in F$.
\end{itemize}

A filter $F$ of $A$ is said to be {\em proper} iff $F\neq A$, which is equivalent to the fact that $0\notin F$. A proper filter $P$ of $A$ is called a {\em prime filter} iff, for all $a,b\in A$, if $a\vee b\in P$, then $a\in P$ or $b\in P$. A maximal element of the set of all proper filters of $A$ is called a {\em maximal filter}.

If $F$ is a filter of $A$, then the congruence $\equiv (\mod F)$ associated to $F$ is defined by: for all $a,b\in A$, $a\equiv b(\mod F)$ iff $d_A(a,b)\in F$. It is obvious that $a\equiv b(\mod F)$ iff $a\rightarrow b\in F$ and $b\rightarrow a\in F$. We recall that residuated lattices form an equational class, which ensures us that the quotient set $A/_{\equiv (\mod F)}$ is a residuated lattice, which we denote by $A/F$. For all $a\in A$, we will denote by $a/F$ the congruence class of $A$ with respect to $\equiv (\mod F)$. It is easily seen that: $a/F=1/F$ iff $a\in F$ (\cite{haj}).

A subset $F$ of $A$ is a filter iff $1\in F$ and, for all $a,b\in A$, $a\in F$ and $a\rightarrow b\in F$ imply $b\in F$.

\begin{lemma}{\rm \cite{kow}, \cite{pic}} A proper filter $F$ of $A$ is maximal iff, for all $a\in A\setminus F$, there exists a nonzero natural number $n$ such that $\neg \, (a^{n})\in F$.
\label{l2.5}
\end{lemma}

$A$ is said to be {\em simple} iff it has exactly two filters.

\begin{lemma}{\rm \cite{kow}, \cite{pic}} $A$ is simple iff, for all $a\in A\setminus \{1\}$, there exists a nonzero natural number $n$ such that $a^{n}=0$.
\label{l2.6}
\end{lemma}

If $s:A\rightarrow L$ is a function, then by the {\em kernel of $s$} we will understand the set $\{a\in A|s(a)=1\}$, which we will denote ${\rm Ker}(s)$. Notice that, if $s$ is a residuated lattice morphism, then: $s$ is injective iff ${\rm Ker}(s)=\{1\}$.
 
\section{Generalized Bosbach States}
\label{bosbach}

\hspace*{10pt} In this section we will present two generalizations for the Bosbach states defined on residuated lattices. We will start from the observation that in the definition of Bosbach states we report essentially to the MV-algebra structure of $[0,1]$. By writing the axioms of Bosbach states in different equivalent ways, there will result two distinct ways of generalizing Bosbach states when we replace the standard MV-algebra $[0,1]$ with an arbitrary residuated lattice.

Throughout this section, let $A$ be a residuated lattice.

\begin{proposition}{\rm \cite{gg}, \cite{lcciu}} Let $s:A\rightarrow [0,1]$ be a function such that $s(0)=0$ and $s(1)=1$. Then the following are equivalent:
\begin{enumerate}
\item\label{p3.1(1)} for all $a,b\in A$, $1+s(a\wedge b)=s(a\vee b)+s(d_A(a,b))$;
\item\label{p3.1(2)} for all $a,b\in A$, $1+s(a\wedge b)=s(a)+s(a\rightarrow b)$;
\item\label{p3.1(3)} for all $a,b\in A$, $s(a)+s(a\rightarrow b)=s(b)+s(b\rightarrow a)$.
\end{enumerate}

\label{p3.1}
\end{proposition}

The proposition above has been proven in \cite{fla} for Bosbach states defined on pseudo-BL-algebras. Then it was extended to more general cases (\cite{lcciu1}, \cite{dvura}, \cite{kuhr}).

\begin{definition}
A {\em Bosbach state} on $A$ is a function $s:A\rightarrow [0,1]$ such that $s(0)=0$, $s(1)=1$ and $s$ verifies the equivalent conditions from Proposition \ref{p3.1}.
\label{d3.2}
\end{definition}

\begin{lemma}{\rm \cite{gg}, \cite{lcciu}} Let $s:A\rightarrow [0,1]$ be a Bosbach state. Then, for all $a,b\in A$, we have:
\begin{enumerate}
\item\label{l3.2(1)} $s(\neg \, a)=1-s(a)$;
\item\label{l3.2(2)} $s$ is order-preserving: $a\leq b$ implies $s(a)\leq s(b)$;
\item\label{l3.2(3)} $s(a)+s(b)=s(a\vee b)+s(a\wedge b)$.
\end{enumerate}
\label{l3.2}
\end{lemma}

Let $s:A\rightarrow [0,1]$ be a Bosbach state. Then, from the fact that $s$ is order-preserving and from Lemma \ref{l2.2}, (\ref{l2.2(5)}) and (\ref{l2.2nenum}), we deduce that, for all $a,b\in A$:
\begin{enumerate}
\item\label{rnenum(1)} $1-s(a\vee b)+s(a\wedge b)=s(a\vee b)\rightarrow _{L}s(a\wedge b)$ (because $s(a\wedge b)\leq s(a\vee b)$);

\item\label{rnenum(2)} $1-s(d_A(a,b))+s(a\wedge b)=s(d_A(a,b))\rightarrow _{L}s(a\wedge b)$ (because $s(a\wedge b)\leq s(d_A(a,b))$);
\item\label{rnenum(3)} $1-s(a)+s(a\wedge b)=s(a)\rightarrow _{L}s(a\wedge b)$ (because $s(a\wedge b)\leq s(a)$);

\item\label{rnenum(4)} $1-s(a\rightarrow b)+s(a\wedge b)=s(a\rightarrow b)\rightarrow _{L}s(a\wedge b)$ (because $s(a\wedge b)\leq s(a\rightarrow b)$);
\item\label{rnenum(5)} $1-s(a\rightarrow b)+s(b)=s(a\rightarrow b)\rightarrow _{L}s(b)$ (because $s(b)\leq s(a\rightarrow b)$).
\end{enumerate}

It follows:

\begin{itemize}
\item condition (\ref{p3.1(1)}) of Proposition \ref{p3.1} is equivalent to each of the following two equalities:

\noindent $(1^{\prime })$ for all $a,b\in A$, $s(d_A(a,b))=s(a\vee b)\rightarrow _{L}s(a\wedge b)$;

\noindent $(1^{\prime \prime })$ for all $a,b\in A$, $s(a\vee b)=s(d_A(a,b))\rightarrow _{L}s(a\wedge b)$;

\item condition (\ref{p3.1(2)}) of Proposition \ref{p3.1} is equivalent to each of the following two equalities:

\noindent $(2^{\prime })$ for all $a,b\in A$, $s(a\rightarrow b)=s(a)\rightarrow _{L}s(a\wedge b)$;

\noindent $(2^{\prime \prime })$ for all $a,b\in A$, $s(a)=s(a\rightarrow b)\rightarrow _{L}s(a\wedge b)$;

\item condition (\ref{p3.1(3)}) of Proposition \ref{p3.1} is equivalent to the following equality:

\noindent $(3^{\prime })$ for all $a,b\in A$, $s(a\rightarrow b)\rightarrow _{L}s(b)=s(b\rightarrow a)\rightarrow _{L}s(a)$.
\end{itemize}

Each of the equalities $(1^{\prime })$, $(1^{\prime \prime })$, $(2^{\prime })$, $(2^{\prime \prime })$ and $(3^{\prime })$ can suggest a way to extend the definition of the Bosbach state when the standard MV-algebra $[0,1]$ is replaced by an arbitrary residuated lattice. First, we shall compare these conditions in the general case when the codomain of $s$ is an arbitrary residuated lattice.

In the following, let $(L,\vee ,\wedge ,\odot ,\rightarrow ,0,1)$ be a residuated lattice and $s:A\rightarrow L$ be an arbitrary function.

\begin{proposition}
If $s(0)=0$ and $s(1)=1$, then the following are equivalent:

\begin{enumerate}
\item\label{p3.3(1)} for all $a,b\in A$, $s(d_A(a,b))=s(a\vee b)\rightarrow s(a\wedge b)$;
\item\label{p3.3(2)} for all $a,b\in A$ with $b\leq a$, $s(a\rightarrow b)=s(a)\rightarrow s(b)$;
\item\label{p3.3(3)} for all $a,b\in A$, $s(a\rightarrow b)=s(a)\rightarrow s(a\wedge b)$;
\item\label{p3.3(4)} for all $a,b\in A$, $s(a\rightarrow b)=s(a\vee b)\rightarrow s(b)$.
\end{enumerate}
\label{p3.3}
\end{proposition}
\begin{proof}
Let $a,b\in A$.

\noindent (\ref{p3.3(1)})$\Rightarrow $(\ref{p3.3(2)}): Assume $b\leq a$. Then, by Lemma \ref{l2.2}, (\ref{l2.2(3)}), we have that $d_A(a,b)=a\rightarrow b$, so $s(a\rightarrow b)=s(d_A(a,b))=s(a\vee b)\rightarrow s(a\wedge b)=s(a)\rightarrow s(b)$.

\noindent (\ref{p3.3(2)})$\Rightarrow $(\ref{p3.3(1)}): By Lemma \ref{l2.2}, (\ref{l2.2(9)}) and (\ref{l2.2(3)}), $(a\vee b)\rightarrow (a\wedge b)=(a\rightarrow (a\wedge b))\wedge (b\rightarrow (a\wedge b))=(a\rightarrow a)\wedge (a\rightarrow b)\wedge (b\rightarrow a)\wedge (b\rightarrow b)=(a\rightarrow b)\wedge (b\rightarrow a)=d_A(a,b)$, and $a\wedge b\leq a\vee b$, so $s(d_A(a,b))=s((a\vee b)\rightarrow (a\wedge b))=s(a\vee b)\rightarrow s(a\wedge b)$.

\noindent (\ref{p3.3(2)})$\Rightarrow $(\ref{p3.3(3)}): By Lemma \ref{l2.2}, (\ref{l2.2(9)}) and (\ref{l2.2(3)}), $a\rightarrow (a\wedge b)=a\rightarrow b$, and $a\wedge b\leq a$, so $s(a\rightarrow b)=s(a\rightarrow (a\wedge b))=s(a)\rightarrow s(a\wedge b)$.

\noindent (\ref{p3.3(3)})$\Rightarrow $(\ref{p3.3(2)}): Trivial.

\noindent (\ref{p3.3(2)})$\Leftrightarrow $(\ref{p3.3(4)}): Analogous to (\ref{p3.3(2)})$\Leftrightarrow $(\ref{p3.3(3)}).\end{proof}

\begin{proposition}
If $s(0)=0$ and $s(1)=1$, then the following are equivalent:

\begin{enumerate}
\item\label{p3.4(1)} for all $a,b\in A$, $s(a\vee b)=s(d_A(a,b))\rightarrow s(a\wedge b)$;
\item\label{p3.4(2)} for all $a,b\in A$, $s(a)=s(a\rightarrow b)\rightarrow s(a\wedge b)$;

\item\label{p3.4(3)} for all $a,b\in A$ with $b\leq a$, $s(a)=s(a\rightarrow b)\rightarrow s(b)$;
\item\label{p3.4(4)} for all $a,b\in A$, $s(a\vee b)=s(a\rightarrow b)\rightarrow s(b)$;
\item\label{p3.4(5)} for all $a,b\in A$, $s(a\rightarrow b)\rightarrow s(b)=s(b\rightarrow a)\rightarrow s(a)$.
\end{enumerate}
\label{p3.4}
\end{proposition}
\begin{proof}

\noindent (\ref{p3.4(1)})$\Rightarrow $(\ref{p3.4(3)}): If $b\leq a$, then $a\vee b=a$, $d_A(a,b)=a\rightarrow b$ (by Lemma \ref{l2.2}, (\ref{l2.2(3)})) and $a\wedge b=b$.

\noindent (\ref{p3.4(3)})$\Rightarrow $(\ref{p3.4(1)}): As in the proof of Proposition \ref{p3.3}, $d_A(a,b)=(a\vee b)\rightarrow (a\wedge b)$, and, since $a\wedge b\leq a\vee b$, we have $s(a\vee b)=s((a\vee b)\rightarrow (a\wedge b))\rightarrow s(a\wedge b)=s(d_A(a,b))\rightarrow s(a\wedge b)$.

\noindent (\ref{p3.4(2)})$\Rightarrow $(\ref{p3.4(3)}): Trivial.

\noindent (\ref{p3.4(3)})$\Rightarrow $(\ref{p3.4(2)}): Since $a\wedge b\leq a$ and, as in the proof of Proposition \ref{p3.3}, $a\rightarrow (a\wedge b)=a\rightarrow b$, we have: $s(a)=s(a\rightarrow (a\wedge b))\rightarrow s(a\wedge b)=s(a\rightarrow b)\rightarrow s(a\wedge b)$. 

\noindent (\ref{p3.4(3)})$\Leftrightarrow $(\ref{p3.4(4)}): Analogous to the proof of (\ref{p3.4(3)})$\Leftrightarrow $(\ref{p3.4(2)}).

\noindent (\ref{p3.4(4)})$\Rightarrow $(\ref{p3.4(5)}): $s(a\rightarrow b)\rightarrow s(b)=s(a\vee b)=s(b\vee a)=s(b\rightarrow a)\rightarrow s(a)$.

\noindent (\ref{p3.4(5)})$\Rightarrow $(\ref{p3.4(2)}): If $b\leq a$, then, by Lemma \ref{l2.2}, (\ref{l2.2(2)}) and (\ref{l2.2(3)}), $s(a)=1\rightarrow s(a)=s(1)\rightarrow s(a)=s(b\rightarrow a)\rightarrow s(a)=s(a\rightarrow b)\rightarrow s(b)$.\end{proof}

Propositions \ref{p3.3} and \ref{p3.4} suggest the following generalizations of Bosbach states:

\begin{definition}
$s$ is called a {\em generalized Bosbach state of type I} (or, in brief, a {\em state of type I} or {\em a type I state}) iff it verifies the equivalent conditions from Proposition \ref{p3.3}.

$s$ is called a {\em generalized Bosbach state of type II} (or, in brief, a {\em state of type II} or {\em a type II state}) iff it verifies the equivalent conditions from Proposition \ref{p3.4}.

$s$ is called a {\em generalized Bosbach state of type III} (or, in brief, a {\em state of type III} or {\em a type III state}) iff it is both a generalized Bosbach state of type I and a generalized Bosbach state of type II.
\label{defsbg}
\end{definition}

\begin{example}
Any residuated lattice morphism $s:A\rightarrow L$ is an order-preserving type I state. The identity morphism $id_A:A\rightarrow A$ is a type II state iff $A$ is an MV-algebra.

Indeed, any residuated lattice morphism verifies condition (\ref{p3.3(2)}) from Proposition \ref{p3.3}. For the remark concerning the identity morphism see Corollary \ref{c3.13}.
\label{sbgex1}
\end{example}

\begin{example}
In \cite[Definition 3.1]{ciudvuhyc}, the notion of state-operator on a BL-algebra is introduced. Condition (\ref{p3.3(3)}) from Proposition \ref{p3.3} is exactly axiom (2) from this definition, thus any state-operator is a type I state. Moreover, according to \cite[Lemma 3.5, (c)]{ciudvuhyc}, any state-operator is an order-preserving type I state.

\end{example}

\begin{example}
Let $A$ be a Heyting algebra and $a\in A$. We denote by $s_a:A\rightarrow A$ the function defined by: for all $x\in A$, $s_a(x)=a\rightarrow x$. For all $x,y\in A$, $s_a(x)\rightarrow s_a(x\wedge y)=(a\rightarrow x)\rightarrow (a\rightarrow (x\wedge y))=(a\wedge (a\rightarrow x))\rightarrow (x\wedge y)=(a\wedge x)\rightarrow (x\wedge y)=((a\wedge x)\rightarrow x)\wedge ((a\wedge x)\rightarrow y)=1\wedge ((a\wedge x)\rightarrow y)=(a\wedge x)\rightarrow y=a\rightarrow (x\rightarrow y)=s_a(x\rightarrow y)$, by Lemma \ref{l2.2}, (\ref{l2.2(7)}), a property of Heyting algebras from Section \ref{preliminaries} and Lemma \ref{l2.2}, (\ref{l2.2(9)}) and (\ref{l2.2(3)}). Thus $s_a$ is an order-preserving type I state, by Lemma \ref{l2.2}, (\ref{l2.2(4)}) and Proposition \ref{p3.3}, (\ref{p3.3(3)}).
\label{sbgex2}
\end{example}

\begin{example}
Let $(A,\leq ,0,1)$ be a bounded chain. By denoting, for all $x,y\in A$, $x\wedge y=\inf \{x,y\}$, $x\vee y=\sup \{x,y\}$ and $x\rightarrow y=\begin{cases}1, & x\leq y,\\ y, & x>y,\end{cases}$ $(A,\vee ,\wedge ,0,1)$ becomes a Heyting algebra. The verification is immediate; this is an example of Heyting algebra from \cite{bal}.

Let $a\in A\setminus \{0\}$, $[0,a)=\{x\in A|x<a\}$ and $f:[0,a)\rightarrow A$ a strictly order-preserving function with $f(0)=0$. We consider the function $f_a:A\rightarrow A$, defined by: for all $x\in A$, $f_a(x)=\begin{cases}f(x), & x<a,\\ 1, & x\geq a.\end{cases}$ Then $f_a$ is an order preserving type I state. Indeed, $f_a$ is obviously order-preserving and let $x,y\in A$ with $y\leq x$. We have to prove that $f_a(x\rightarrow y)=f_a(x)\rightarrow f_a(y)$, which is clear for $x=y$, as Lemma \ref{l2.2}, (\ref{l2.2(3)}) shows. So let $y<x$ now. Since $x\rightarrow y=y$, we have to prove that $f_a(y)=f_a(x)\rightarrow f_a(y)$. We have three cases:

\begin{itemize}
\item $y<x<a$. Then $f_a(x)\rightarrow f_a(y)=f(x)\rightarrow f(y)=f(y)=f_a(y)$.
\item $y<a\leq x$. Then $f_a(x)\rightarrow f_a(y)=1\rightarrow f_a(y)=f(y)=f_a(y)$, by Lemma \ref{l2.2}, (\ref{l2.2(2)}).
\item $a\leq y<x$. Then $f_a(x)\rightarrow f_a(y)=1\rightarrow 1=1=f_a(y)$, by Lemma \ref{l2.2}, (\ref{l2.2(2)}).
\end{itemize}

So $f_a$ is a type I state, by Proposition \ref{p3.3}, (\ref{p3.3(2)}).

Now assume that the chain $A$ is a complete lattice and let $s:A\rightarrow A$ be an arbitrary strictly order-preserving type I state. We denote $a=\inf \{x\in A|s(x)=1\}$ and let $0\leq y<x<a$. Then $s(y)=s(x)\rightarrow s(y)$ and $s(y)<1$, so, by the law of residuation and Lemma \ref{l2.2}, (\ref{l2.2nenum2}), it follows that $s(y)<s(x)$. Therefore, $f=s\mid _{[0,a)}:[0,a)\rightarrow A$ is strictly order-preserving and, obviously, $s=f_a$.
\label{sbgex3}
\end{example}

\begin{example}
Let $(A,\leq ,0,1)$, with the Heyting algebra structure from Example \ref{sbgex3}. Let $s:A\rightarrow A$ be a type II state. Then, for all $x,y\in A$ with $y<x$, $s(x)=s(x\rightarrow y)\rightarrow s(y)=s(y)\rightarrow s(y)=1$, by Proposition \ref{p3.4}, (\ref{p3.4(3)}) and Lemma \ref{l2.2}, (\ref{l2.2(3)}). So $s(x)=\begin{cases}0, & x=0,\\ 1, & x>0.\end{cases}$
\label{sbgex4}
\end{example}

\begin{proposition}
If $s$ is a generalized Bosbach state of type I, then, for all $a,b\in A$:

\begin{enumerate}

\item\label{p3.6(1)} $s(\neg \, a)=\neg \, s(a)$;
\item\label{p3.6(2)} $s(a\vee b)\rightarrow s(a)=s(b)\rightarrow s(a\wedge b)$;
\item\label{p3.6(3)} $s((a\rightarrow b)\rightarrow b)=s(a\rightarrow b)\rightarrow s(b)$;
\item\label{p3.6(4)} $s((a\rightarrow b)\rightarrow b)=(s(a\vee b)\rightarrow s(b))\rightarrow s(b)$;
\item\label{p3.6(5)} $s(a\vee b)\rightarrow (s(a)\wedge s(b))=(s(a)\vee s(b))\rightarrow s(a\wedge b)$;
\item\label{p3.6(6)} $s(a)\odot s(a\rightarrow (a\odot b))\leq s(a\odot b)$.

\end{enumerate}

\label{p3.6}
\end{proposition}
\begin{proof}
\noindent (\ref{p3.6(1)}): $s(\neg \, a)=s(a\rightarrow 0)=s(a)\rightarrow s(0)=s(a)\rightarrow 0=\neg \, s(a)$ (see Proposition \ref{p3.3}, (\ref{p3.3(2)})).

\noindent (\ref{p3.6(2)}): By Proposition \ref{p3.3}, (\ref{p3.3(3)}) and (\ref{p3.3(4)}).

\noindent (\ref{p3.6(3)}): By Proposition \ref{p3.3}, (\ref{p3.3(2)}), and Lemma \ref{l2.2}, (\ref{l2.2nenum}).

\noindent (\ref{p3.6(4)}): By (\ref{p3.6(3)}) and Proposition \ref{p3.3}, (\ref{p3.3(4)}), $s((a\rightarrow b)\rightarrow b)=s(a\rightarrow b)\rightarrow s(b)=(s(a\vee b)\rightarrow s(b))\rightarrow s(b)$.

\noindent (\ref{p3.6(5)}): By Lemma \ref{l2.2}, (\ref{l2.2(9)}) and Proposition \ref{p3.6}, (\ref{p3.6(2)}), $s(a\vee b)\rightarrow (s(a)\wedge s(b))=(s(a\vee b)\rightarrow s(a))\wedge (s(a\vee b)\rightarrow s(b))=(s(b)\rightarrow s(a\wedge b))\wedge (s(a)\rightarrow s(a\wedge b))=(s(a)\vee s(b))\rightarrow s(a\wedge b)$.

\noindent (\ref{p3.6(6)}): By Lemma \ref{l2.2}, (\ref{l2.2(5)}), Proposition \ref{p3.3}, (\ref{p3.3(2)}) and Lemma \ref{l2.2}, (\ref{l2.2(6)}), $a\odot b\leq a$, so $s(a\rightarrow (a\odot b))=s(a)\rightarrow s(a\odot b)$, hence $s(a)\odot s(a\rightarrow (a\odot b))=s(a)\odot (s(a)\rightarrow s(a\odot b))\leq s(a\odot b)$.\end{proof}

\begin{remark}
In the case when $L$ is the standard MV-algebra $[0,1]$, order-preserving type I states $s:A\rightarrow [0,1]$ coincide with Bosbach states on $A$, as the identity (\ref{p3.3(3)}) from Proposition \ref{p3.3} is equivalent to the identity $(2^{\prime })$, and type II states $s:A\rightarrow [0,1]$ coincide with Bosbach states on $A$, as the identity (\ref{p3.4(5)}) from Proposition \ref{p3.4} is equivalent to the identity $(3^{\prime })$.
\label{r3.7}
\end{remark}

\begin{remark}
Let $B$ be a residuated lattice, $s:B\rightarrow L$ be a function and $f:A\rightarrow L$ be a residuated lattice morphism. Then, by Proposition \ref{p3.3}, (\ref{p3.3(3)}), if $s$ is a type I state, then $s\circ f:A\rightarrow L$ is a type I state, and if, moreover, $s$ is order-preserving and $f$ is order-preserving, then $s\circ f$ is order-preserving. By Proposition \ref{p3.4}, (\ref{p3.4(2)}), if $s$ is a type II state, then $s\circ f$ is a type II state. Thus, if $s$ is a type III state, then $s\circ f$ is a type III state.
\label{rmorf}
\end{remark}

\begin{proposition}
Let $s:A\rightarrow L$ be an order-preserving type I state. Then, for all $a,b,x,y\in A$:

\begin{enumerate}
\item\label{p3.8(1)} $s(a)\odot s(b)\leq s(a\odot b)$;
\item\label{p3.8(2)} $s(a)\ominus s(b)\leq s(a\ominus b)$;
\item\label{p3.8(3)} $s(a\rightarrow b)\leq s(a)\rightarrow s(b)$;
\item\label{p3.8(4)} $s(a\rightarrow b)\odot s(b\rightarrow a)\leq d_L(s(a),s(b))$;

\item\label{p3.8(5)} $s(d_A(a,b))\leq d_L(s(a),s(b))$;
\item\label{p3.8(6)} $s(d_A(a,x))\odot s(d_A(b,y))\leq d_L(s(d_A(a,b)),s(d_A(x,y)))$.
\end{enumerate}
\label{p3.8}
\end{proposition}
\begin{proof}
\noindent (\ref{p3.8(1)}) By the law of residuation, the fact that $s$ is order-preserving, Lemma \ref{l2.2}, (\ref{l2.2(5)}), Proposition \ref{p3.3}, (\ref{p3.3(2)}) and again the law of residuation, $b\leq a\rightarrow (a\odot b)$, hence $s(b)\leq s(a\rightarrow (a\odot b))=s(a)\rightarrow s(a\odot b)$, therefore $s(a)\odot s(b)\leq s(a\odot b)$.

\noindent (\ref{p3.8(2)}) By Proposition \ref{p3.6}, (\ref{p3.6(1)}) and (\ref{p3.8(1)}) from the current proposition, $s(a)\ominus s(b)=s(a)\odot \neg \, s(b)=s(a)\odot s(\neg \, b)\leq s(a\odot \neg \, b)=s(a\ominus b)$.

\noindent (\ref{p3.8(3)}) By Proposition \ref{p3.3}, (\ref{p3.3(3)}) and Lemma \ref{l2.2}, (\ref{l2.2(4)}), $s(a\rightarrow b)\leq s(a)\rightarrow s(a\wedge b)\leq s(a)\rightarrow s(b)$.

\noindent (\ref{p3.8(4)}) By (\ref{p3.8(3)}) from the current proposition and (\ref{l2.2nenum2}) and (\ref{l2.2(5)}) from Lemma \ref{l2.2}, $s(a\rightarrow b)\odot s(b\rightarrow a)\leq (s(a)\rightarrow s(b))\odot (s(b)\rightarrow s(a))\leq d_L(s(a),s(b))$. 

\noindent (\ref{p3.8(5)}) By the fact that $s$ is order-preserving and (\ref{p3.8(3)}), $s(d_A(a,b))=s((a\rightarrow b)\wedge (b\rightarrow a))\leq s(a\rightarrow b)\wedge s(b\rightarrow a)\leq (s(a)\rightarrow s(b))\wedge (s(b)\rightarrow s(a))=d_L(s(a),s(b))$.

\noindent (\ref{p3.8(6)}) By (\ref{p3.8(1)}) and (\ref{p3.8(5)}) from the current proposition, along with Lemma \ref{lA}, (\ref{lA(5)}), $s(d_A(a,x))\odot s(d_A(b,y))\leq s(d_A(a,x)\odot d_A(b,y))\leq s(d_A(d_A(a,b),d_A(x,y)))\leq d_L(s(d_A(a,b)),s(d_A(x,y)))$.\end{proof}

\begin{proposition}
Let $s:A\rightarrow L$ be a type II state. Then, for all $a,b\in A$:

\begin{enumerate}
\item\label{p3.9(1)} $b\leq a$ implies $s(b)\leq s(a)$ (that is $s$ is order-preserving);
\item\label{p3.9(2)} $s(a)=\neg \, s(\neg \, a)$;
\item\label{p3.9(3)} $s(\neg \, \neg \, a)=s(a)=\neg \, \neg \, s(a)$;
\item\label{p3.9(4)} $s(a\rightarrow b)=s((a\rightarrow b)\rightarrow b)\rightarrow s(b)$;
\item\label{p3.9(6)} $s(\neg \, a)=\neg \, s(a)$;
\item\label{p3.9(5)} $s(a\odot b)=\neg \, s(a\rightarrow \neg \, b)$.
\end{enumerate}

\label{p3.9}
\end{proposition}
\begin{proof}
\noindent (\ref{p3.9(1)}) By Lemma \ref{l2.2}, (\ref{l2.2nenum}) and Proposition \ref{p3.4}, (\ref{p3.4(3)}), $s(b)\leq s(a\rightarrow b)\rightarrow s(b)=s(a)$.

\noindent (\ref{p3.9(2)}) By Proposition \ref{p3.4}, (\ref{p3.4(3)}), $s(a)=s(a\rightarrow 0)\rightarrow 0=s(\neg \, a)\rightarrow 0=\neg \, s(\neg \, a)$.

\noindent (\ref{p3.9(3)}) By (\ref{p3.9(2)}) and Lemma \ref{l2.3}, (\ref{l2.3(2)}), $s(\neg \, \neg \, a)=\neg \, s(\neg \, \neg \, \neg \, a)=\neg \, s(\neg \, a)=s(a)$, and also $\neg \, \neg \, s(a)=\neg \, \neg \, \neg \, s(\neg \, a)=\neg \, s(\neg \, a)=s(a)$.

\noindent (\ref{p3.9(4)}) By Lemma \ref{l2.2}, (\ref{l2.2nenum}) and Proposition \ref{p3.4}, (\ref{p3.4(3)}).

\noindent (\ref{p3.9(6)}) By (\ref{p3.9(4)}) and (\ref{p3.9(3)}), $s(\neg \, a)=s(a\rightarrow 0)=s((a\rightarrow 0)\rightarrow 0)\rightarrow s(0)=s(\neg \, \neg \, a)\rightarrow 0=\neg \, s(\neg \, \neg \, a)=\neg \, s(a)$.

\noindent (\ref{p3.9(5)}) By (\ref{p3.9(2)}) and Lemma \ref{l2.3}, (\ref{l2.3(5)}), $s(a\odot b)=\neg \, s(\neg \, (a\odot b))=\neg \, s(a\rightarrow \neg \, b)$.\end{proof}

\begin{remark}
Let $A$ be a totally ordered product algebra and $s:A\rightarrow A$ a type II state. Since, for all $a\in A$, $a\wedge \neg \, a\in \{a,\neg \, a\}$, it follows that, for all $a\in A\setminus \{0\}$, $\neg \, a=0$. By Lemma \ref{l2.3}, (\ref{0neg1}), for all $a\in A$, $s(\neg \, a)=\begin{cases}s(1)=1, & {\rm if}\ a=0,\\ s(0)=0, & {\rm if}\ a\neq 0.\end{cases}$

By Proposition \ref{p3.9}, (\ref{p3.9(6)}) and Lemma \ref{l2.3}, (\ref{0neg1}), it follows that there exists a unique type II state $s:A\rightarrow A$, namely, for all $a\in A$, $s(a)=\neg \, s(\neg \, a)=\begin{cases}0, & {\rm if}\ a=0,\\ 1, & {\rm if}\ a\neq 0,\end{cases}$ as this is indeed a type II state, by Proposition \ref{p3.4}, (\ref{p3.4(3)}) and Lemma \ref{l2.2}, (\ref{l2.2(2)}) and (\ref{l2.2(3)}).
\end{remark}

\begin{remark}
\label{r3.6}
In general, if $s:A\rightarrow L$ is a state of type I, then it is not necessarily order-preserving (that is: $a,b\in A$ and $a\leq b$ do not necessarily imply $s(a)\leq s(b)$) and, even if it is order-preserving, it is not necessarily a state of type II.

Indeed, let us consider the following example of residuated lattice from {\rm \cite{kow}, \cite{gal}, \cite{ior}}: $A=\{0,a,b,c,d,1\}$, with the following partial order relation and operations:

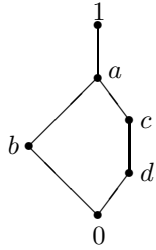
\begin{center}
\begin{picture}(60,100)(0,0)
\put(37,11){\circle*{3}}
\put(35,0){$0$}
\put(37,11){\line(3,4){12}}
\put(49,27){\circle*{3}}
\put(53,24){$d$}

\put(49,27){\line(0,1){20}}

\put(49,47){\circle*{3}}
\put(53,44){$c$}

\put(49,47){\line(-3,4){12}}
\put(37,63){\circle*{3}}
\put(41,63){$a$}
\put(37,11){\line(-1,1){26}}
\put(11,37){\circle*{3}}
\put(3,34){$b$}
\put(11,37){\line(1,1){26}}

\put(37,63){\line(0,1){20}}
\put(37,83){\circle*{3}}
\put(35,85){$1$}
\end{picture}

\end{center}

\begin{center}
\begin{tabular}{cc}
\begin{tabular}{c|cccccc}

$\rightarrow $ & $0$ & $a$ & $b$ & $c$ & $d$ & $1$ \\ \hline
$0$ & $1$ & $1$ & $1$ & $1$ & $1$ & $1$ \\
$a$ & $0$ & $1$ & $b$ & $c$ & $c$ & $1$ \\
$b$ & $c$ & $1$ & $1$ & $c$ & $c$ & $1$ \\
$c$ & $b$ & $1$ & $b$ & $1$ & $a$ & $1$ \\
$d$ & $b$ & $1$ & $b$ & $1$ & $1$ & $1$ \\
$1$ & $0$ & $a$ & $b$ & $c$ & $d$ & $1$
\end{tabular}

& \hspace*{11pt}
\begin{tabular}{c|cccccc}

$\odot $ & $0$ & $a$ & $b$ & $c$ & $d$ & $1$ \\ \hline
$0$ & $0$ & $0$ & $0$ & $0$ & $0$ & $0$ \\
$a$ & $0$ & $a$ & $b$ & $d$ & $d$ & $a$ \\
$b$ & $0$ & $b$ & $b$ & $0$ & $0$ & $b$ \\
$c$ & $0$ & $d$ & $0$ & $d$ & $d$ & $c$ \\
$d$ & $0$ & $d$ & $0$ & $d$ & $d$ & $d$ \\

$1$ & $0$ & $a$ & $b$ & $c$ & $d$ & $1$
\end{tabular}
\end{tabular}
\end{center}

Let us determine the generalized Bosbach states $s_i:A\rightarrow A$. The type I states from $A$ to $A$ are:

\begin{center}
\begin{tabular}{c|cccccc}
$x$ & $0$ & $a$ & $b$ & $c$ & $d$ & $1$\\ \hline 
$s_1(x)$ & $0$ & $a$ & $0$ & $1$ & $a$ & $1$\\
$s_2(x)$ & $0$ & $a$ & $b$ & $c$ & $d$ & $1$\\ 
$s_3(x)$ & $0$ & $1$ & $0$ & $1$ & $1$ & $1$\\ 

$s_4(x)$ & $0$ & $1$ & $b$ & $c$ & $c$ & $1$\\ 

$s_5(x)$ & $0$ & $1$ & $c$ & $b$ & $b$ & $1$\\ 
$s_6(x)$ & $0$ & $1$ & $1$ & $0$ & $0$ & $1$ 
\end{tabular}
\end{center}

Out of these, the only order-preserving ones are $s_2$, $s_3$, $s_4$, $s_5$ and $s_6$. Indeed, $s_1$ is not order-preserving, as $c\leq a$ and $s_1(c)=1>s_1(a)=a$.

The type II states from $A$ to $A$ are $s_3$, $s_4$, $s_5$ and $s_6$.

\end{remark}

\begin{proposition}
Let $A$ and $L$ be divisible residuated lattices and $s:A\rightarrow L$ an order-preserving type I state. Then, for all $a,b\in A$:

\begin{enumerate}
\item\label{p3.10(1)} $s(a\odot b)=s(a)\odot s(a\rightarrow (a\odot b))$;
\item\label{p3.10(2)} $s(a\wedge b)=s(a)\odot s(a\rightarrow b)$.
\end{enumerate}
\label{p3.10}
\end{proposition}

\begin{proof}

\noindent (\ref{p3.10(1)}) By Lemma \ref{l2.2}, (\ref{l2.2(5)}) and Proposition \ref{p3.3}, (\ref{p3.3(2)}), $s(a)\odot s(a\rightarrow (a\odot b))=s(a)\odot (s(a)\rightarrow s(a\odot b))=s(a)\wedge s(a\odot b)=s(a\odot b)$.

\noindent (\ref{p3.10(2)}) $a\rightarrow (a\odot (a\rightarrow b))=a\rightarrow (a\wedge b)=(a\rightarrow a)\wedge (a\rightarrow b)=a\rightarrow b$, by (\ref{l2.2(9)}) and (\ref{l2.2(3)}) from Lemma \ref{l2.2}.\end{proof}

\begin{lemma}
Let $A$ be a residuated lattice, $L$ an MV-algebra and $s:A\rightarrow L$ an order-preserving type I state. Then, for all $a,b\in A$, $s(a\vee b)=s((a\rightarrow b)\rightarrow b)=s((b\rightarrow a)\rightarrow a)$.
\label{l3.11}
\end{lemma}
\begin{proof}
By Proposition \ref{p3.6}, (\ref{p3.6(4)}), Lemma \ref{l2.4} and Lemma \ref{l2.2}, (\ref{l2.2(3)}) and (\ref{l2.2(2)}), $s((a\rightarrow b)\rightarrow b)=(s(a\vee b)\rightarrow s(b))\rightarrow s(b)=(s(b)\rightarrow s(a\vee b))\rightarrow s(a\vee b)=1\rightarrow s(a\vee b)=s(a\vee b)=s(b\vee a)=s((b\rightarrow a)\rightarrow a)$.\end{proof}

\begin{proposition}
Let $A$ be a residuated lattice, $L$ an MV-algebra and $s:A\rightarrow L$ a function. Then the following are equivalent:

\begin{enumerate}
\item\label{p3.12(1)} $s$ is an order-preserving type I state;

\item\label{p3.12(2)} $s$ is a type II state.
\end{enumerate}
\label{p3.12}
\end{proposition}
\begin{proof}

\noindent (\ref{p3.12(1)})$\Rightarrow $(\ref{p3.12(2)}): Let $s$ be an order-preserving type I state. By Proposition \ref{p3.6}, (\ref{p3.6(3)}) and Lemma \ref{l3.11}, for all $a,b\in A$, $s(a\rightarrow b)\rightarrow s(b)=s((a\rightarrow b)\rightarrow b)=s((b\rightarrow a)\rightarrow a)=s(b\rightarrow a)\rightarrow s(a)$, hence, by Proposition \ref{p3.4}, (\ref{p3.4(5)}), $s$ is a type II state.

\noindent (\ref{p3.12(2)})$\Rightarrow $(\ref{p3.12(1)}): Let $s$ be a type II state and $a,b\in A$. Then, by Proposition \ref{p3.4}, (\ref{p3.4(2)}), $s(a\rightarrow b)\rightarrow s(a\wedge b)=s(a)$, therefore, by Lemma \ref{l2.4}, Lemma \ref{l2.2}, (\ref{l2.2(5)}), Proposition \ref{p3.9}, (\ref{p3.9(1)}) and Lemma \ref{l2.2}, (\ref{l2.2(3)}) and (\ref{l2.2(2)}), $s(a)\rightarrow s(a\wedge b)=(s(a\rightarrow b)\rightarrow s(a\wedge b))\rightarrow s(a\wedge b)=(s(a\wedge b)\rightarrow s(a\rightarrow b))\rightarrow s(a\rightarrow b)=1\rightarrow s(a\rightarrow b)=s(a\rightarrow b)$. Hence, by Proposition \ref{p3.3}, (\ref{p3.3(3)}), $s$ is a type I state. By Proposition \ref{p3.9}, \ref{p3.9(1)}, $s$ is also order-preserving.\end{proof}

\begin{corollary}
Let $A$ be a residuated lattice. Then the following are equivalent:

\begin{enumerate}
\item\label{c3.13(1)} $A$ is an MV-algebra;
\item\label{c3.13(2)} any order-preserving type I state $s:A\rightarrow A$ is a type II state.
\end{enumerate}
\label{c3.13}
\end{corollary}
\begin{proof}
\noindent (\ref{c3.13(1)})$\Rightarrow $(\ref{c3.13(2)}): By Proposition \ref{p3.12}.

\noindent (\ref{c3.13(2)})$\Rightarrow $(\ref{c3.13(1)}): The identity $id_A:A\rightarrow A$ obviously is an order-preserving type I state. Hence it is also a type II state, and this condition on the identity is sufficient for this implication to take place. By Proposition \ref{p3.4}, (\ref{p3.4(5)}), for all  $a,b\in A$, we have: $(a\rightarrow b)\rightarrow b=(b\rightarrow a)\rightarrow a$. By Lemma \ref{l2.4}, it follows that $A$ is an MV-algebra.\end{proof}

\begin{proposition}
Let $s:A\rightarrow L$ be a type III state. Then, for all $a,b\in A$, $s((a\rightarrow b)\rightarrow b)=s((b\rightarrow a)\rightarrow a)$.
\label{p3.14}
\end{proposition}
\begin{proof}
By Proposition \ref{p3.6}, (\ref{p3.6(3)}) and Proposition \ref{p3.4}, (\ref{p3.4(5)}), $s((a\rightarrow b)\rightarrow b)= s(a\rightarrow b)\rightarrow s(b)=s(b\rightarrow a)\rightarrow s(a)=s((b\rightarrow a)\rightarrow a)$.\end{proof}

\begin{openproblem}
In Proposition \ref{p3.14}, do we have $s((a\rightarrow b)\rightarrow b)=s((b\rightarrow a)\rightarrow a)=s(a\vee b)$?
\end{openproblem}

\begin{proposition}

Let $A$ be an MV-algebra, $L$ a residuated lattice and $s:A\rightarrow L$ a function such that $s(0)=0$ and $s(1)=1$. Then the following are equivalent:

\begin{enumerate}
\item\label{pnoua(1)} $s$ is an order-preserving type I state;

\item\label{pnoua(2)} for all $a,b\in A$, we have:

\noindent (a) $s(\neg \, a)=\neg \, s(a)$;

\noindent (b) $s(a\rightarrow b)\rightarrow (s(a)\rightarrow s(b))=1$;

\noindent (c) $s(a\oplus b)=(s(a)\rightarrow s(a\odot b))\rightarrow s(b)$.

\end{enumerate}
\label{pnoua}

\end{proposition}
\begin{proof}
\noindent (\ref{pnoua(1)})$\Rightarrow $(\ref{pnoua(2)}): Let $s:A\rightarrow L$ be an order-preserving type I state. (a) is Proposition \ref{p3.6}, (\ref{p3.6(1)}) and (b) results from Proposition \ref{p3.8}, (\ref{p3.8(3)}) and Lemma \ref{l2.2}, (\ref{l2.2(3)}).

Let us prove (c) now. By Lemma \ref{mvlema}, (\ref{mvdemorgan2}), (\ref{mvdemorgan1}), (\ref{mvdistrib}) and (\ref{mvsum}), for all $a,b\in A$, we have: $(a\rightarrow (a\odot b))\rightarrow b=\neg \, (\neg \, a\oplus (a\odot b))\oplus b=(\neg \, \neg \, a\odot \neg \, (a\odot b))\oplus b=(\neg \, \neg \, a\odot (\neg \, a\oplus \neg \, b))\oplus b=(a\odot (\neg \, a\oplus \neg \, b))\oplus b= (a\wedge \neg \, b)\oplus b=(a\oplus b)\wedge (\neg \, b\oplus b)=a\oplus b$. But, by Lemma \ref{l2.2}, (\ref{l2.2(5)}) and the law of residuation, $a\odot b\leq a$ and $b\leq a\rightarrow (a\odot b)$, hence, by Proposition \ref{p3.3}, (\ref{p3.3(2)}), $s(a\oplus b)=s((a\rightarrow (a\odot b))\rightarrow b)=s(a\rightarrow (a\odot b))\rightarrow s(b)=(s(a)\rightarrow s(a\odot b))\rightarrow s(b)$.

\noindent (\ref{pnoua(2)})$\Rightarrow $(\ref{pnoua(1)}): Assume that $s$ satisfies (a), (b) and (c). (b) immediately implies that $s$ is order-preserving, as shown by Lemma \ref{l2.2}, (\ref{l2.2(3)}). Now let $a,b\in A$ with $b\leq a$, thus $\neg \, a\odot b=0$, by Lemma \ref{l2.3}, (\ref{l2.3(2)}) and (\ref{l2.3(1)}). By applying (a) and (c) we obtain: $s(a\rightarrow b)=s(\neg \, a\oplus b)=(s(\neg \, a)\rightarrow s(\neg \, a\odot b))\rightarrow s(b)=(s(\neg \, a)\rightarrow s(0))\rightarrow s(b)=(s(\neg \, a)\rightarrow 0)\rightarrow s(b)=\neg \, s(\neg \, a)\rightarrow s(b)=s(\neg \, \neg \, a)\rightarrow s(b)=s(a)\rightarrow s(b)$. Therefore $s$ is an order-preserving type I state.\end{proof}

\begin{remark}
Conditions (a)-(c) from the previous proposition represent the algebraic form of the axioms $(FP_1)-(FP_3)$ from \cite[page 327]{fla1} in the context of probabilistic many-valued logic FP(\L $_k$,\L ), where \L $_k$ is the $k$-valued \L ukasiewicz logic and \L is the infinite-valued \L ukasiewicz logic.
\label{r3.26}
\end{remark}

\begin{remark}
Notice that in the example from Remark \ref{r3.6} all type II states from $A$ to $A$ are type I states. This is the case for all the numerous examples of finite residuated lattices we considered, whose generalized Bosbach states we determined by means of a small computer program, including the cases where the domain was different from the codomain.

In addition to that, it can be easily proven that, for any pair of residuated lattices $A$ and $L$ which are each determined by one of the three fundamental continuous t-norms, all type II states from $A$ to $L$ are type I states.

However, we have been unable to prove this in the general case and therefore we mention it as an open problem.
\end{remark}

\begin{openproblem}
Prove that, if $s:A\rightarrow L$ is a type II state, then $s$ is a type I state.
\end{openproblem}

Obviously, the definition of type I and type II states can be extended to non-commutative residuated lattices, pseudo-BCK-algebras, pseudo-hoops and so on. It remains to be investigated, for each o these cases, to what extent an interesting theory of generalized Bosbach states can be developped.

\section{Properties of Generalized Bosbach States}
\label{filtcan}

\hspace*{10pt} In this section we study properties of the quotient residuated lattice $A/{\rm Ker}(s)$, where ${\rm Ker}(s)$ is the canonical filter associated with a (type I or type II) generalized Bosbach state $s:A\rightarrow L$. We introduce the notion of state-morphism in our context, then the state-morphisms are characterized in terms of ${\rm Ker}(s)$ and $A/{\rm Ker}(s)$.

Let $A$ and $L$ be two nontrivial residuated lattices.

\begin{lemma}
Let $s:A\rightarrow L$ be an order-preserving type I state or a type II state. Then ${\rm Ker}(s)$ is a proper filter of $A$.
\label{l4.1}
\end{lemma}

\begin{proof}
Obviously, $1\in {\rm Ker}(s)$ and $0\notin {\rm Ker}(s)$. Now let $a,b\in A$ such that $a,a\rightarrow b\in {\rm Ker}(s)$, that is $s(a)=s(a\rightarrow b)=1$. We must prove that $b\in {\rm Ker}(s)$, that is $s(b)=1$.

If $s$ is an order-preserving type I state, then, by Proposition \ref{p3.3}, (\ref{p3.3(3)}) and Lemma \ref{l2.2}, (\ref{l2.2(2)}), $1=s(a\rightarrow b)=s(a)\rightarrow s(a\wedge b)=1\rightarrow s(a\wedge b)=s(a\wedge b)\leq s(b)$, thus $s(b)=1$.

If $s$ is a type II state, then, by Lemma \ref{l2.2}, (\ref{l2.2(2)}), Proposition \ref{p3.4}, (\ref{p3.4(5)}) and Lemma \ref{l2.2}, (\ref{l2.2(3)}), $s(b)=1\rightarrow s(b)=s(a\rightarrow b)\rightarrow s(b)=s(b\rightarrow a)\rightarrow s(a)=s(b\rightarrow a)\rightarrow 1=1$.\end{proof}

\begin{lemma}
Let $s:A\rightarrow L$ be an order-preserving type I state or a type II state and $a,b\in A$. If $a/{\rm Ker}(s)=b/{\rm Ker}(s)$, then $s(a)=s(b)=s(a\vee b)=s(a\wedge b)$. 
\label{l4.3}
\end{lemma}

\begin{proof}
Assume $a/{\rm Ker}(s)=b/{\rm Ker}(s)$, that is $d_A(a,b)\in {\rm Ker}(s)$, which means that $s(d_A(a,b))=1$.

If $s$ is an order-preserving type I state, then, by Proposition \ref{p3.3}, (\ref{p3.3(1)}) and Lemma \ref{l2.2}, (\ref{l2.2(3)}), $1=s(d_A(a,b))=s(a\vee b)\rightarrow s(a\wedge b)$, so $s(a\vee b)\leq s(a\wedge b)$. But $s(a\wedge b)\leq s(a),s(b)\leq s(a\vee b)$, as $s$ is order-preserving. Therefore $s(a)=s(b)=s(a\vee b)=s(a\wedge b)$.

If $s$ is a type II state, then, by Proposition \ref{p3.4}, (\ref{p3.4(1)}) and Lemma \ref{l2.2}, (\ref{l2.2(2)}), $s(a\vee b)=s(d_A(a,b))\rightarrow s(a\wedge b)=1\rightarrow s(a\wedge b)=s(a\wedge b)$. But, by Proposition \ref{p3.9}, (\ref{p3.9(1)}), $s(a\wedge b)\leq s(a),s(b)\leq s(a\vee b)$. Therefore $s(a)=s(b)=s(a\vee b)=s(a\wedge b)$.\end{proof}

\begin{proposition}

Let $s:A\rightarrow L$ be an order-preserving type I state and $a,b\in A$. Then the following are equivalent:

\begin{enumerate}
\item\label{p4.4(1)} $a/{\rm Ker}(s)=b/{\rm Ker}(s)$;
\item\label{p4.4(2)} $s(a\vee b)=s(a\wedge b)$;
\item\label{p4.4(3)} $s(a)=s(b)=s(a\vee b)$;
\item\label{p4.4(4)} $s(a)=s(b)=s(a\wedge b)$.
\end{enumerate}
\label{p4.4}
\end{proposition}
\begin{proof}
The implications (\ref{p4.4(1)})$\Rightarrow $(\ref{p4.4(2)}),(\ref{p4.4(3)}),(\ref{p4.4(4)}) result from Lemma \ref{l4.3}, and the implications, (\ref{p4.4(2)})$\Rightarrow $(\ref{p4.4(3)}),(\ref{p4.4(4)}) result from the fact that $s$ is order-preserving.

\noindent (\ref{p4.4(3)})$\Rightarrow $(\ref{p4.4(4)}): By Lemma \ref{l2.2}, (\ref{l2.2(3)}), and Proposition \ref{p3.6}, (\ref{p3.6(2)}), $1=s(a\vee b)\rightarrow s(a)=s(b)\rightarrow s(a\wedge b)$, hence $s(b)\leq s(a\wedge b)$. But $s(a\wedge b)\leq s(b)$, as $s$ is order-preserving. So that $s(b)=s(a\wedge b)$.

\noindent (\ref{p4.4(3)})$\Rightarrow $(\ref{p4.4(4)}): By Proposition \ref{p3.6}, (\ref{p3.6(2)}) and Lemma \ref{l2.2}, (\ref{l2.2(3)}), $s(a\vee b)\rightarrow s(a)=s(b)\rightarrow s(a\wedge b)=1$, thus $s(a\vee b)\leq s(a)$. But $s$ is order-preserving and so $s(a)\leq s(a\vee b)$. Hence $s(a)=s(a\vee b)$.

\noindent (\ref{p4.4(3)})$\Rightarrow $(\ref{p4.4(1)}): By Proposition \ref{p3.3}, (\ref{p3.3(4)}) and Lemma \ref{l2.2}, (\ref{l2.2(3)}), $s(a\rightarrow b)=s(a\vee b)\rightarrow s(b)=1$, so $a\rightarrow b\in {\rm Ker}(s)$. Analogously, $b\rightarrow a\in {\rm Ker}(s)$. Thus $a/{\rm Ker}(s)=b/{\rm Ker}(s)$.\end{proof}

Let $s:A\rightarrow L$ be an order-preserving type I state, respectively a type II state. We consider the quotient residuated lattice $A/{\rm Ker}(s)$. By Lemma \ref{l4.3}, we can define a function $\overline{s}:A/{\rm Ker}(s)\rightarrow L$, for all $a\in A$, $\overline{s}(a/{\rm Ker}(s))=s(a)$. It easily follows that $\overline{s}$ is an order-preserving type I state, respectively a type II state.

\begin{proposition}
Assume that the residuated lattice $L$ is involutive and $s:A\rightarrow L$ is an order-preserving type I state. Then $A/{\rm Ker}(s)$ is involutive.
\label{p4.5}

\end{proposition}
\begin{proof}
Let $a\in A$. By Proposition \ref{p3.6}, (\ref{p3.6(1)}), $s(\neg \, \neg \, a)=\neg \, \neg \, s(a)=s(a)$. By Lemma \ref{l2.3}, (\ref{l2.3(2)}), $a\vee \neg \, \neg \, a=\neg \, \neg \, a$, so $s(a\vee \neg \, \neg \, a)=s(\neg \, \neg \, a)$. It follows that $s(a\vee \neg \, \neg \, a)=s(\neg \, \neg \, a)=s(a)$, so, by Proposition \ref{p4.4}, $\neg \, \neg \, a/{\rm Ker}(s)=a/{\rm Ker}(s)$, thus $A/{\rm Ker}(s)$ is involutive.\end{proof}

\begin{corollary}
Assume that $A$ is divisible, $L$ is involutive and $s:A\rightarrow L$ is an order-preserving type I state. Then $A/{\rm Ker}(s)$ is an MV-algebra.
\label{c4.6}

\end{corollary}
\begin{proof}
It is easily seen that $A/{\rm Ker}(s)$ is divisible, and, by Proposition \ref{p4.5}, it is also involutive, hence it is an MV-algebra.\end{proof}

\begin{proposition}
Let $A$ be an MTL-algebra, $L$ an MV-algebra and $s:A\rightarrow L$ an order-preserving type I state. Then $A/{\rm Ker}(s)$ is an MV-algebra.
\label{p4.7}
\end{proposition}

\begin{proof}
Let $a,b\in A$. By Lemma \ref{l3.11}, $s(a\vee b)=s((a\rightarrow b)\rightarrow b)=s((b\rightarrow a)\rightarrow a)$. Let $x=(a\rightarrow b)\rightarrow b$ and $y=(b\rightarrow a)\rightarrow a$. By Lemma \ref{lMTL}, $a\vee b=x\wedge y$. It follows, by Lemma \ref{l3.11}, that $s(x)=s(y)=s(x\wedge y)$. By Proposition \ref{p4.4} and Lemma \ref{l2.4}, $x/{\rm Ker}(s)=y/{\rm Ker}(s)$, therefore $A/{\rm Ker}(s)$ is an MV-algebra.\end{proof}

\begin{proposition}
If $s:A\rightarrow L$ is a type III state, then $A/{\rm Ker}(s)$ is involutive.
\label{p4.8}
\end{proposition}
\begin{proof}
Let $a\in A$. By Proposition \ref{p3.9}, (\ref{p3.9(3)}), $s(a)=s(\neg \, \neg \, a)$ and, by Lemma \ref{l2.3}, (\ref{l2.3(2)}), $a\vee \neg \, \neg \, a=\neg \, \neg \, a$, hence $s(a\vee \neg \, \neg \, a)=s(\neg \, \neg \, a)=s(a)$. By Proposition \ref{p4.4}, $\neg \, \neg \, a/{\rm Ker}(s)=a/{\rm Ker}(s)$, thus $A/{\rm Ker}(s)$ is involutive.\end{proof}

Let $s:A\rightarrow L$ be an arbitrary function. Let us consider the properties:

\noindent $(\alpha )$ for all $a,b\in A$, $s(a\vee b)=s(a)\vee s(b)$;

\noindent $(\beta )$ for all $a,b\in A$, $s(a\wedge b)=s(a)\wedge s(b)$;

\noindent $(\gamma )$ for all $a,b\in A$, $s(a\rightarrow b)=s(a)\rightarrow s(b)$;

\noindent $(\delta )$ for all $a,b\in A$, $s(a\odot b)=s(a)\odot s(b)$.

\begin{lemma}
Assume that $s:A\rightarrow L$ is an order-preserving type I state. Then each of the conditions $(\alpha )$ and $(\beta )$ implies $(\gamma )$.
\label{l4.9}
\end{lemma}
\begin{proof}
\noindent $(\alpha )\Rightarrow (\gamma )$ By Proposition \ref{p3.3}, (\ref{p3.3(4)}) and Lemma \ref{l2.2}, (\ref{l2.2(9)}) and (\ref{l2.2(3)}), $s(a\rightarrow b)=s(a\vee b)\rightarrow s(b)=(s(a)\vee s(b))\rightarrow s(b)=(s(a)\rightarrow s(b))\wedge (s(b)\rightarrow s(b))=s(a)\rightarrow s(b)$.

\noindent $(\beta )\Rightarrow (\gamma )$ By Proposition \ref{p3.3}, (\ref{p3.3(3)}) and Lemma \ref{l2.2}, (\ref{l2.2(9)}) and (\ref{l2.2(3)}), $s(a\rightarrow b)=s(a)\rightarrow s(a\wedge b)=s(a)\rightarrow (s(a)\wedge s(b))=(s(a)\rightarrow s(b))\wedge (s(b)\rightarrow s(b))=s(a)\rightarrow s(b)$.\end{proof}

\begin{lemma}
Let $L$ be an involutive residuated lattice and $s:A\rightarrow L$ an order-preserving type I state. Then $(\beta )$ implies $(\alpha )$.
\label{l4.10}
\end{lemma} 
\begin{proof}
Let $a,b\in A$. By Proposition \ref{p3.6}, (\ref{p3.6(1)}) and Lemma \ref{l2.2}, (\ref{l2.2(9)}), $\neg \, s(a\vee b)=s(\neg \, (a\vee b))=s(\neg \, a\wedge \neg \, b)=s(\neg \, a)\wedge s(\neg \, b)=\neg \, s(a)\wedge \neg \, s(b)=\neg \, (s(a)\vee s(b))$, hence $\neg \, \neg \, s(a\vee b)=\neg \, \neg \, (s(a)\vee s(b))$, so that $s(a\vee b)=s(a)\vee s(b)$, since $L$ is involutive.\end{proof}

\begin{proposition}
Let $L$ be an MV-algebra and $s:A\rightarrow L$ an order-preserving type I state. Then conditions $(\alpha )$ and $(\gamma )$ are equivalent.
\label{p4.11}
\end{proposition}
\begin{proof}
\noindent $(\alpha )\Rightarrow (\gamma )$ By Lemma \ref{l4.9}.

\noindent $(\gamma )\Rightarrow (\alpha )$ Let $a,b\in A$. By Lemma \ref{l3.11} and Lemma \ref{l2.4}, $s(a\vee b)=s((a\rightarrow b)\rightarrow b)=(s(a)\rightarrow s(b))\rightarrow s(b)=s(a)\vee s(b)$.\end{proof}

\begin{lemma}
Let $s:A\rightarrow L$ be an order-preserving type I state. Then:

\begin{enumerate}
\item\label{l4.12(1)} if $(\gamma )$ then, for all $a,b\in A$, $\neg \, s(a\odot b)=\neg \, (s(a)\odot s(b))$;
\item\label{l4.12(2)} if $L$ is involutive, then $(\gamma )$ implies $(\delta )$.
\end{enumerate}
\label{l4.12}
\end{lemma}
\begin{proof}
\noindent (\ref{l4.12(1)}) Let $a,b\in A$. By Lemma \ref{l2.3}, (\ref{l2.3(5)}), $\neg \, (a\odot b)=a\rightarrow \neg \, b$. Thus, by Proposition \ref{p3.6}, (\ref{p3.6(1)}), $\neg \, s(a\odot b)=s(\neg \, (a\odot b))=s(a\rightarrow \neg \, b)=s(a)\rightarrow s(\neg \, b)=s(a)\rightarrow \neg \, s(b)=\neg \, (s(a)\odot s(b))$.

\noindent (\ref{l4.12(2)}) By (\ref{l4.12(1)}).\end{proof}

\begin{corollary}
Let $A$ be a divisible residuated lattice, $L$ be an MV-algebra and $s:A\rightarrow L$ an order-preserving type I state. Then $(\alpha )$, $(\beta )$ and $(\gamma )$ are equivalent.
\label{c4.13}
\end{corollary}
\begin{proof}
By Proposition \ref{p4.11}, $(\alpha )\Leftrightarrow (\gamma )$. By Lemma \ref{l4.9}, $(\beta )\rightarrow (\gamma )$. It remains to show:

\noindent $(\gamma )\Rightarrow (\beta )$ Let $a,b\in A$. By Lemma \ref{l4.12}, $s(a\wedge b)=s(a\odot (a\rightarrow b))=s(a)\odot s(a\rightarrow b)=s(a)\odot (s(a)\rightarrow s(b))=s(a)\wedge s(b)$.\end{proof}

\begin{definition}
A function $s:A\rightarrow L$ is called a {\em state-morphism} iff it fulfills $(\alpha )$, $(\beta )$, $(\gamma )$, $s(0)=0$ and $s(1)=1$.
\label{d4.12}
\end{definition}

\begin{remark}

Any state-morphism is an order-preserving type I state.
\label{rsm}

\end{remark}
\begin{proof}

By Proposition \ref{p3.3}, (\ref{p3.3(2)}), any state-morphism is a type I state. By $(\alpha )$ and $(\beta )$, it is also a lattice morphism, thus an order-preserving function.\end{proof}

If $L$ is the standard MV-algebra $[0,1]$, then Definition \ref{d4.12} coincides with the concept of state-morphism from \cite{dindvu}, \cite{dvura} etc.. 

\begin{proposition}
Let $s:A\rightarrow L$ be an order-preserving type I state. If $A/{\rm Ker}(s)$ is totally ordered, then $s$ is a state-morphism.
\label{p4.13}
\end{proposition}
\begin{proof}
Let $a,b\in A$. Then $a/{\rm Ker}(s)\leq b/{\rm Ker}(s)$ or $b/{\rm Ker}(s)\leq a/{\rm Ker}(s)$. Assume, for example, that $a/{\rm Ker}(s)\leq b/{\rm Ker}(s)$, thus $(a\rightarrow b)/{\rm Ker}(s)=1/{\rm Ker}(s)$, that is $a\rightarrow b\in {\rm Ker}(s)$, that is $s(a\rightarrow b)=1$, by Lemma \ref{l2.2}, (\ref{l2.2(3)}). By Remark \ref{rsm}, Proposition \ref{p3.3}, (\ref{p3.3(3)}) and (\ref{p3.3(4)}) and Lemma \ref{l2.2}, (\ref{l2.2(3)}), $1=s(a\rightarrow b)=s(a)\rightarrow s(a\wedge b)=s(a\vee b)\rightarrow s(b)$, thus $s(a)\leq s(a\wedge b)$ and $s(a\vee b)\leq s(b)$. Since $s$ is order-preserving, it follows that $s(a)=s(a\wedge b)\leq s(a\vee b)=s(b)$, thus $s(a\vee b)=s(a)\vee s(b)$ and $s(a\wedge b)=s(a)\wedge s(b)$. By Lemma \ref{l4.9}, we also have $s(a\rightarrow b)=s(a)\rightarrow s(b)$, therefore $s$ is a state-morphism.\end{proof}

\begin{corollary}
Let $s:A\rightarrow L$ be an order-preserving type I state. If $A/{\rm Ker}(s)$ is an MV-algebra and ${\rm Ker}(s)$ is a maximal filter of $A$, then $s$ is a state-morphism. 
\label{c4.14}
\end{corollary}
\begin{proof}

If ${\rm Ker}(s)$ is a maximal filter of $A$, then $A/{\rm Ker}(s)$ is a simple MV-algebra, thus totally ordered (see \cite{cito2}). By Proposition \ref{p4.13}, it follows that $s$ is a state-morphism.\end{proof}

\begin{proposition}
Assume that $L$ is totally ordered and $s:A\rightarrow L$ is a state-morphism. Then $A/{\rm Ker}(s)$ is totally ordered.
\label{p4.15}
\end{proposition}
\begin{proof}
Let $a,b\in A$. Then $s(a)\leq s(b)$ or $s(b)\leq s(a)$, so that, by Lemma \ref{l2.2}, (\ref{l2.2(3)}), $s(a\rightarrow b)=s(a)\rightarrow s(b)=1$ or $s(b\rightarrow a)=s(b)\rightarrow s(a)=1$, thus $a/{\rm Ker}(s)\leq b/{\rm Ker}(s)$ or $b/{\rm Ker}(s)\leq a/{\rm Ker}(s)$.\end{proof}

\begin{corollary}
Let $L$ be totally ordered and $s:A\rightarrow L$ be an order-preserving type I state. Then: $s$ is a state-morphism iff $A/{\rm Ker}(s)$ is totally ordered.
\label{c4.16}
\end{corollary}
\begin{proof}
By Propositions \ref{p4.13} and \ref{p4.15}.\end{proof}

\begin{corollary}
If $L$ is totally ordered and $s:A\rightarrow L$ is a state-morphism then ${\rm Ker}(s)$ is a prime filter of $A$.
\end{corollary}
\begin{proof}

By Proposition \ref{p4.15}, $A/{\rm Ker}(s)$ is totally ordered, thus, by \cite[Proposition 1.41, (iii)]{pic}, ${\rm Ker}(s)$ is a prime filter.\end{proof}

\begin{corollary}
If $A$ is an MTL-algebra, $L$ is totally ordered and $s:A\rightarrow L$ an order-preserving type I state, then: $s$ is a state-morphism iff ${\rm Ker}(s)$ is a prime filter.
\end{corollary}

\begin{proof}
Apply Corollary \ref{c4.16} and \cite[Lemma 2.61]{beloh}.\end{proof}

\begin{proposition}
Let $L$ be a simple residuated lattice and $s:A\rightarrow L$ a state-morphism. Then ${\rm Ker}(s)$ is a maximal filter of $A$.
\label{p4.17}
\end{proposition}
\begin{proof}
Let $a\in A\setminus {\rm Ker}(s)$, thus $s(a)\neq 1$. By Lemmas \ref{l4.1} and \ref{l2.5}, it is sufficient to prove that there exists an $n\in \N ^{*}$ such that $\neg \, (a^{n})\in {\rm Ker}(s)$. By Lemma \ref{l2.6}, there exists an $n\in \N ^{*}$ such that $(s(a))^n=0$. By Lemma \ref{l4.12}, (\ref{l4.12(1)}), $s(\neg \, (a^n))=s(a^n\rightarrow 0)=s(a^n)\rightarrow s(0)=s(a^n)\rightarrow 0=\neg \, s(a^n)=\neg \, s(a)^n=1$, so $\neg \, (a^{n})\in {\rm Ker}(s)$.\end{proof}

\begin{remark}
It is known that the standard MV-algebra $[0,1]$ is simple (\cite{ciudvuhyc}). Then from Proposition \ref{p4.17} we get the following known result (\cite{dvura2}, \cite{lcciu1}): a Bosbach state $s:A\rightarrow [0,1]$ is a state-morphism iff ${\rm Ker}(s)$ is a maximal filter in $A$.
\label{r4.18}
\end{remark}

\section{Glivenko Property and Rie\v can states}
\label{riecan}

\hspace*{10pt} In this section we study the relation between generalized Bosbach states on a residuated lattice $A$ with Glivenko property and generalized Bosbach states on the involutive residuated lattice ${\rm Reg}(A)$ of the regular elements of $A$. We define the notion of generalized Rie\v can state and we relate type I states and generalized Rie\v can states.

In the following, let $A$ be a residuated lattice and ${\rm Reg}(A)=\{\neg \, a|a\in A\}=\{a\in A|a=\neg \, \neg \, a\}$ the set of the {\em regular elements} of $A$. $A$ is said to be {\em involutive} iff $A={\rm Reg}(A)$. For all $a,b\in A$, we denote $a\vee ^{*}b=\neg \, \neg \, (a\vee b)$, $a\wedge ^{*}b=\neg \, \neg \, (a\wedge b)$, $a\odot ^{*}b=\neg \, \neg \, (a\odot b)$.

We say that $A$ has {\em Glivenko property} iff, for all $a,b\in A$, $\neg \, \neg \, (a\rightarrow b)=a\rightarrow \neg \, \neg \, b$.

\begin{proposition}{\rm \cite[Theorem 2.1, page 163]{cito2}} The following are equivalent:

\noindent (i) $A$ has Glivenko property;\index{residuated lattice!Glivenko}

\noindent (ii) $({\rm Reg}(A),\vee ^{*},\wedge ^{*},\odot ^{*},\rightarrow ,0,1)$ is an involutive residuated lattice and $\neg \, \neg \, :A\rightarrow {\rm Reg}(A)\ :\ a\rightarrow \neg \, \neg \, a$ is a surjective morphism of residuated lattices.\index{${\rm Reg}(A)$}
\label{`glivnegneg`}
\end{proposition}

Heyting algebras and BL-algebras have Glivenko property.

Until mentioned otherwise, let $A$ be a residuated lattice with Glivenko property. We define $\varphi =\neg \, \neg \, $ to be the surjective morphism from the proposition above.

Let $L$ be a residuated lattice. If $s:A\rightarrow L$ is a type I state (respectively a type II state), then, obviously, $s\mid _{{\rm Reg}(A)}:{\rm Reg}(A)\rightarrow L$ is a type I state (respectively a type II state).

Let $s:{\rm Reg}(A)\rightarrow L$ be an arbitrary function. We define $\tilde{s}:A\rightarrow L$ by $\tilde{s}(a)=s(\varphi (a))$ for all $a\in A$.

\begin{proposition}
Assume that $A$ has Glivenko property and $L$ is involutive. If $s:{\rm Reg}(A)\rightarrow L$ is a type I state (respectively a type II state), then $\tilde{s}:A\rightarrow L$ is a type I state (respectively a type II state).
\label{p5.1}
\end{proposition}

\begin{proof}
Assume that $s$ is a type I state. Then, for all $a,b\in A$, $\tilde{s}(a\rightarrow b)=s(\varphi (a\rightarrow b))=s(\varphi (a))\rightarrow s(\varphi (b))=s(\varphi (a))\rightarrow s(\varphi (a)\wedge \varphi (b))=s(\varphi (a))\rightarrow s(\varphi (a\wedge b))=\tilde{s}(a)\rightarrow \tilde{s}(a\wedge b)$. So $\tilde{s}$ is a type I state.

Assume that $s$ is a type II state. Then, for all $a,b\in A$, $\tilde{s}(a\rightarrow b)\rightarrow \tilde{s}(b)=s(\varphi (a\rightarrow b))\rightarrow s(\varphi (b))=s(\varphi (a)\rightarrow \varphi (b))\rightarrow s(\varphi (b))=s(\varphi (b)\rightarrow \varphi (a))\rightarrow s(\varphi (a))=s(\varphi (b\rightarrow a))\rightarrow s(\varphi (a))=\tilde{s}(b\rightarrow a)\rightarrow \tilde{s}(a)$. So $\tilde{s}$ is a type II state.\end{proof}

\begin{remark}
Let $s_1:A\rightarrow L$ be a type I state (respectively a type II state). By applying Proposition \ref{p3.6}, (\ref{p3.6(1)}) (respectively Proposition \ref{p3.9}, (\ref{p3.9(3)})), we obtain, for all $a\in A$, $s_1(a)=s_1(\varphi (a))$. Then, if $s:{\rm Reg}(A)\rightarrow L$ is a type I state (respectively a type II state), it follows that $\tilde{s}:A\rightarrow L$ is the unique type I state (respectively the unique type II state) such that $\tilde{s}\mid _{{\rm Reg}(A)}=s$.
\label{r5.2}
\end{remark}

In the following, let $A$ be an arbitrary residuated lattice. On the set $A$ we introduce the binary operation $\oplus $  by: for all $a,b\in A$, $a\oplus b=\neg \, a\rightarrow \neg \, \neg \, b=\neg \, b\rightarrow \neg \, \neg \, a$ (see Lemma \ref{l2.3}, (\ref{l2.3(4)}) and (\ref{l2.3(2)})).

\begin{lemma}{\rm \cite{kuhr}, Lemma 3.6.2} For all $a,b,c\in A$, we have:

\begin{enumerate}

\item\label{l5.3(1)} $a\oplus 0=\neg \, \neg \, a$;
\item\label{l5.3(2)} $a\oplus 1=a$;
\item\label{l5.3(3)} $\oplus $ is associative and commutative;
\item\label{l5.3(4)} if $a\leq b$ then $a\oplus c\leq b\oplus c $;
\item\label{l5.3(5)} $a\vee b\leq a\oplus b$;
\item\label{l5.3(6)} $a\oplus b=\neg \, \neg \, (a\oplus b)=\neg \, \neg \, a\oplus \neg \, \neg \, b$.
\end{enumerate}
\label{l5.3}
\end{lemma}

For all $a,b\in A$, we denote $a\perp b$ iff $\neg \, \neg \, a\leq \neg \, b$ iff $\neg \, \neg \, b\leq \neg \, a$ (see Lemma \ref{l2.3}, (\ref{l2.3(3)}) and (\ref{l2.3(2)})).

A {\em Rie\v can state on $A$} is a function $m:A\rightarrow [0,1]$ such that $m(1)=1$ and, for all $a,b\in A$ with $a\perp b$, $m(a\oplus b)=m(a)+m(b)$.

\begin{lemma}{\rm \cite{gg}, \cite{dvura}, \cite{lcciu}} If $m$ is a Rie\v can state on $A$, then:

\begin{enumerate}
\item\label{l5.4(1)} for all $a\in A$, $m(\neg \, a)=1-m(a)$; 
\item\label{l5.4(2)} $m(0)=0$;
\item\label{l5.4(3)} $m$ is order-preserving.
\end{enumerate}
\label{l5.4}
\end{lemma}

Rie\v can states on pseudo-BL-algebras have been defined in \cite{gg}, by generalizing a notion of state on BL-algebras that had been introduced by Rie\v can in \cite{br2}. Later, Rie\v can states on more general structures have been studied (\cite{dvura2}, \cite{lcciu}, \cite{kuhr}, \cite{turmer}).

In what follows we shall extend the notion of Rie\v can state to the context of this paper and we shall point out the relation between the notion we shall obtain and generalized Bosbach states. 

In the following, let $A$ and $L$ be residuated lattices.

\begin{definition}
A function $m:A\rightarrow L$ is called a {\em generalized Rie\v can state} iff the following conditions are verified, for all $a,b\in A$:

\noindent (a) $m(1)=1$;

\noindent (b) if $a\perp b$, then $m(a)\perp m(b)$;

\noindent (c) if $a\perp b$, then $m(a\oplus b)=m(a)\oplus m(b)$.
\label{d5.5}
\end{definition}

\begin{proposition}
Let $m:A\rightarrow L$ be a generalized Rie\v can state. Then, for all $a,b\in A$, we have:

\begin{enumerate}
\item\label{p5.6(1)} $\neg \, \neg \, m(\neg \, a)=\neg \, m(a)$; if $L$ is involutive, then $m(\neg \, a)=\neg \, m(a)$ and $m(\neg \, \neg \, a)=m(a)$;
\item\label{p5.6(2)} $m(0)=0$;
\item\label{p5.6(3)} if $b\leq a$, then $\neg \, m(a)\leq \neg \, m(b)$; if $L$ is involutive and $b\leq a$, then $m(b)\leq m(a)$.
\end{enumerate}
\label{p5.6}
\end{proposition}
\begin{proof}
Let $a,b\in A$.

\noindent (\ref{p5.6(1)}) Obviously, $a\perp \neg \, a$, so $m(a)\perp m(\neg \, a)$, that is $\neg \, \neg \, m(\neg \, a)\leq \neg \, m(a)$. Also, by Lemma \ref{l2.2}, (\ref{l2.2(3)}) and Lemma \ref{l2.3}, (\ref{l2.3(2)}), $1=m(1)=m(a\oplus \neg \, a)=m(a)\oplus m(\neg \, a)=\neg \, m(a)\rightarrow \neg \, \neg \, m(\neg \, a)$, thus $\neg \, m(a)\leq \neg \, \neg \, m(\neg \, a)$. Hence $\neg \, \neg \, m(\neg \, a)=\neg \, m(a)$.

\noindent (\ref{p5.6(2)}) Set $a=0$ in (\ref{p5.6(1)}) and apply Lemma \ref{l2.3}, (\ref{0neg1}). 

\noindent (\ref{p5.6(3)}) By Lemma \ref{l2.3}, (\ref{l2.3(3)}), if $b\leq a$ then $b\perp \neg \, a$, so $m(b)\perp m(\neg \, a)$, thus, by (\ref{p5.6(1)}), $\neg \, m(a)=\neg \, \neg \, m(\neg \, a)\leq \neg \, m(b)$.\end{proof}

The next proposition shows that, in the case when $L$ is the standard MV-algebra $[0,1]$, Rie\v can states coincide with generalized Rie\v can states.

\begin{proposition}
Let $m:A\rightarrow [0,1]$ be an arbitrary function. We consider on $[0,1]$ the standard MV-algebra structure. Then: $m$ is a Rie\v can state iff $m$ is a generalized Rie\v can state.
\label{p5.7}
\end{proposition}
\begin{proof}
Let $a,b\in A$.

Assume that $m$ is a Rie\v can state. If $a\perp b$ then $\neg \, \neg \, a\leq \neg \, b$, so, by Lemma \ref{l5.4}, (\ref{l5.4(1)}) and (\ref{l5.4(3)}), $\neg \, \neg \, m(a)=m(a)=m(\neg \, \neg \, a)\leq m(\neg \, b)=1-m(b)=\neg \, m(b)$. Hence $m(a)\perp m(b)$. Thus $m$ is a generalized Rie\v can state.

Now assume that $m$ is a generalized Rie\v can state. If $a\perp b$ then $m(a)\perp m(b)$, so $m(a\oplus b)=m(a)\oplus m(b)=m(a)+m(b)$. Thus $m$ is a Rie\v can state.\end{proof}

\begin{proposition}
Any order-preserving type I state is a generalized Rie\v can state. 
\label{p5.8}
\end{proposition}
\begin{proof}
Let $s:A\rightarrow L$ be an order-preserving type I state and $a,b\in A$ with $a\perp b$. Then $\neg \, \neg \, a\leq \neg \, b$, so, by Proposition \ref{p3.6}, (\ref{p3.6(1)}) and the fact that $s$ is order-preserving, $\neg \, \neg \, s(a)=s(\neg \, \neg \, a)\leq s(\neg \, b)=\neg \, s(b)$. Hence $s(a)\perp s(b)$.

By Proposition \ref{p3.3}, (\ref{p3.3(2)}) and Proposition \ref{p3.6}, (\ref{p3.6(1)}), $s(a\oplus b)=s(\neg \, b\rightarrow \neg \, \neg \, a)=s(\neg \, b)\rightarrow s(\neg \, \neg \, a)=\neg \, s(b)\rightarrow \neg \, \neg \, s(a)=s(a)\oplus s(b)$.

So $s$ is a generalized Rie\v can state.\end{proof}

Obviously, if $A$ has Glivenko property and $m:A\rightarrow L$ is a generalized Rie\v can state, then $m\mid _{{\rm Reg}(A)}:{\rm Reg}(A)\rightarrow L$ is a generalized Rie\v can state.

\begin{proposition}
Assume that $A$ has Glivenko property and $L$ is involutive. Then any generalized Rie\v can state $m:A\rightarrow L$ is an order-preserving type I state.
\label{p5.9}
\end{proposition}
\begin{proof}
Let $m:A\rightarrow L$ be a generalized Rie\v can state and $a,b\in A$ such that $b\leq a$. We show that $m(a\rightarrow b)=m(a)\rightarrow m(b)$.

By Lemma \ref{l2.3}, (\ref{l2.3(3)}), since $b\leq a$, we have that $b\perp \neg \, a$, so $m(b)\perp m(\neg \, a)$. We notice that $\neg \, a\oplus b=\neg \, b\rightarrow \neg \, \neg \, \neg \, a=\neg \, b\rightarrow \neg \, a=a\rightarrow \neg \, \neg \, b$, by Lemma \ref{l2.3}, (\ref{l2.3(2)}) and  (\ref{l2.3(5)}). Since $A$ has Glivenko property and by Lemma \ref{l2.3}, (\ref{l2.3(5)}) and (\ref{l2.3(2)}), $\neg \, \neg \, (a\rightarrow b)=a\rightarrow \neg \, \neg \, b=\neg \, b\rightarrow \neg \, a=\neg \, b\rightarrow \neg \, \neg \, \neg \, a=\neg \, a\oplus b$. By Proposition \ref{p5.6}, (\ref{p5.6(1)}) and the fact that $L$ is involutive, $m(a\rightarrow b)=m(\neg \, \neg \, (a\rightarrow b))=m(\neg \, a\oplus b)=m(\neg \, a)\oplus m(b)=\neg \, m(a)\oplus m(b)=m(a)\rightarrow m(b)$.

So $m$ is an order-preserving type I state.\end{proof}

\begin{remark}

If $A$ has Glivenko property and $L$ is involutive, then, by Propositions \ref{p5.8} and \ref{p5.9}, order-preserving type I states $s:A\rightarrow L$ coincide with generalized Rie\v can states $s:A\rightarrow L$. In particular, if $A$ has Glivenko property and $L$ is the standard MV-algebra, then Bosbach states $s:A\rightarrow L$ coincide with Rie\v can states $s:A\rightarrow L$ (see \cite{dvura2}, \cite{turmer}).
\end{remark}

\begin{proposition}
Not all generalized Rie\v can states are type I or type II states.
\end{proposition}
\begin{proof}
We consider the residuated lattice $A$ from Remark \ref{r3.6}. The generalized Rie\v can states $m:A\rightarrow A$ are the following:

\begin{center}
\begin{tabular}{l|cccccc}

$x$ & $0$ & $a$ & $b$ & $c$ & $d$ & $1$\\ \hline
$s_1(x)$ & $0$ & $a$ & $0$ & $1$ & $a$ & $1$\\

$m_1(x)$ & $0$ & $a$ & $0$ & $1$ & $1$ & $1$\\ 
$m_2(x)$ & $0$ & $a$ & $b$ & $c$ & $c$ & $1$\\ 
$s_2(x)$ & $0$ & $a$ & $b$ & $c$ & $d$ & $1$\\ 

$m_3(x)$ & $0$ & $a$ & $c$ & $b$ & $b$ & $1$\\
$m_4(x)$ & $0$ & $a$ & $1$ & $0$ & $0$ & $1$\\
$m_5(x)$ & $0$ & $1$ & $0$ & $1$ & $a$ & $1$\\ 
$s_3(x)$ & $0$ & $1$ & $0$ & $1$ & $1$ & $1$\\ 
$s_4(x)$ & $0$ & $1$ & $b$ & $c$ & $c$ & $1$\\ 
$m_6(x)$ & $0$ & $1$ & $b$ & $c$ & $d$ & $1$\\ 
$s_5(x)$ & $0$ & $1$ & $c$ & $b$ & $b$ & $1$\\ 
$s_6(x)$ & $0$ & $1$ & $1$ & $0$ & $0$ & $1$
\end{tabular}
\end{center}

As mentioned in Remark \ref{r3.6}, the type I states from $A$ to $A$ are $s_i$, with $i\in \overline{1,6}$, and the type II states from $A$ to $A$ are $s_i$, with $i\in \overline{3,6}$. Out of the generalized Rie\v can states $m_i$, with $i\in \overline{1,6}$, none is a type I or a type II state.\end{proof}

\begin{proposition}
If $A$ is involutive and $s:A\rightarrow L$ is a generalized Rie\v can state such that, for all $a\in A$, $s(\neg \, a)=\neg \, s(a)$, then $s$ is an order-preserving type I state. 
\end{proposition}
\begin{proof}
Let $A$ and $s$ be as in the hypothesis and let $a,b\in A$ such that $b\leq a$. Since $A$ is involutive, it follows that $b=\neg \, \neg \, b$ and $a=\neg \, \neg \, a=\neg \, c$, with $c=\neg \, a$. Thus $\neg \, \neg \, b\leq \neg \, c$, that is $b\perp c$, hence $s(b\oplus c)=s(b)\oplus s(c)$, that is $s(\neg \, c\rightarrow \neg \, \neg \, b)= \neg \, s(c)\rightarrow \neg \, \neg \, s(b)$, that is $s(a\rightarrow b)=s(\neg \, c)\rightarrow s(\neg \, \neg \, b)$, that is $s(a\rightarrow b)=s(a)\rightarrow s(b)$. So, by Proposition \ref{p3.3}, (\ref{p3.3(2)}), $s$ is a type I state. It remains to show that $s(b)\leq s(a)$, which will allow us to conclude that $s$ is order-preserving. We saw that $b\perp c$; it follows that $s(b)\perp s(c)$, which means that $\neg \, \neg \, s(b)\leq \neg \, s(c)$, that is $s(\neg \, \neg \, b)\leq s(\neg \, c)$, that is $s(b)\leq s(a)$.\end{proof}

\begin{corollary}
If $A$ is involutive and $s:A\rightarrow L$ is both a generalized Rie\v can state and a type II state, then $s$ is an order-preserving type I state. 
\end{corollary}

\section{Similarity Convergences and Continuity of States}
\label{convergente}

\hspace*{10pt} The similarity convergence in residuated lattices has been defined in \cite{ggap} based on the biresiduum. In the particular case of MV-algebras, it is dual to the order-convergence, a notion that is defined starting from the distance in MV-algebras. For non-involutive residuated lattices, this duality is not kept, but most part of a good convergence theory (for example, type Cauchy completions) can be obtained.

Starting from the similarity convergence, in this section we introduce three notions of continuity of a generalized Bosbach state and we study the relation between them. If $L$ is a residuated lattice and $E:X^2\rightarrow L$ is an $L$-similarity relation on a nonempty set $X$ (\cite{ggap}), then the similarity convergence on $L$ allows us to define a convergence on $X$ (called $E$-convergence). To an order-preserving type I state $s:A\rightarrow L$ we associate an $L$-similarity relation $\rho_s:A^2\rightarrow L$. The $\rho_s$-convergence is kept by the residuated lattice operations of $A$. Next, working on the $\rho_s$-Cauchy sequences of $A$, we generalize to the context of this paper an important construction from \cite{ioanal3}: the metric completion of an MV-algebra.
 
Until mentioned otherwise, let $X$ be a nonempty set and $L$ a residuated lattice. We recall from \cite{beloh} that an $L$-binary relation on $X$, that is a function $E:X^2\rightarrow L$, is called an {\em $L$-similarity relation on $X$} (or an {\em $L$-equivalence on $X$}) iff, for all $a,b,c\in X$: $E(a,a)=1$, $E(a,b)=E(b,a)$ and $E(a,b)\odot E(b,c)\leq E(a,c)$. An $L$-similarity relation  $E$ on $X$ is called an {\em $L$-equality on $X$} iff, for all $a,b\in X$, $E(a,b)=1$ implies $a=b$. By Lemma \ref{lA}, (\ref{lA(1)}), (\ref{lA(2)}) and (\ref{lA(3)}), $d_L:L^2\rightarrow L$ is an L-equality on $L$.

The fact that a sequence $(c_n)_{n\geq 0}\subseteq L$ is increasing is denoted $(c_n)_{n\geq 0}\uparrow $. The sequence $(c_n)_{n\geq 0}$ is said to be {\em increasing towards $c\in L$} iff $(c_n)_{n\geq 0}\uparrow $ and $\bigvee _{n\geq 0}c_n=c$; this is denoted by $(c_n)_{n\geq 0}\uparrow c$.

The fact that a sequence $(c_n)_{n\geq 0}\subseteq L$ is decreasing is denoted $(c_n)_{n\geq 0}\downarrow $. The sequence $(c_n)_{n\geq 0}$ is said to be {\em decreasing towards $c\in L$} iff $(c_n)_{n\geq 0}\downarrow $ and $\bigwedge _{n\geq 0}c_n=c$; this is denoted by $(c_n)_{n\geq 0}\downarrow c$.

A sequence $(a_n)_{n\geq 0}\subseteq L$ is said to be {\em similarity convergent} (or, in brief, {\em convergent}) {\em towards $a\in L$} iff there exists a sequence $(c_n)_{n\geq 0}\subseteq L$ such that $(c_n)_{n\geq 0}\uparrow 1$ and, for all $n\in \N $, $c_n\leq d_L(a_n,a)$; this is denoted by $\lim _{n\rightarrow \infty }a_n=a$ and $a$ is called the {\em limit of $(a_n)_{n\geq 0}$}. By \cite[Remark 3.7, (i)]{ggap}, the limit of a convergent sequence in a residuated lattice is unique. Obviously, if, for all $n\in \N $, $a_n=\alpha \in L$, then $\lim _{n\rightarrow \infty }a_n=\alpha $. Also, it is obvious that, if $k\in \N $, $a\in L$ and $(b_n)_{n\geq 0}\subseteq L$ such that, for all $n\geq k$, $b_n=a_n$, then: $\lim _{n\rightarrow \infty }a_n=a$ iff $\lim _{n\rightarrow \infty }b_n=a$, as we may take in the definition of the similarity convergence $c_n=0$ for all $n<k$.

The sequence $(a_n)_{n\geq 0}\subseteq L$ is said to be {\em similarity Cauchy} (or, in brief, {\em Cauchy}) iff $\lim _{n,m\rightarrow \infty }d_L(a_n,a_m)=1$, where, naturally, for all $(l_{n,m})_{n,m\geq 0}\subseteq L$, we set $\lim _{n,m\rightarrow \infty }l_{n,m}=\lim _{n\rightarrow \infty }\lim _{m\rightarrow \infty }l_{n,m}$. Any convergent sequence is Cauchy, as shown in \cite{ggap}. $L$ is said to be {\em Cauchy-complete} iff in $L$ any Cauchy sequence is convergent.

\begin{remark}
In \cite{ggap}, a sequence $(a_n)_{n\geq 0}\subseteq L$ is defined to be {\em similarity Cauchy} iff there exists a sequence $(c_n)_{n\geq 0}\subseteq L$ such that $(c_n)_{n\geq 0}\uparrow 1$ and, for all $n,p\in \N $, $c_n\leq d_L(a_n,a_{n+p})$. This is equivalent to our definition, as, for all $(l_n)_{n\geq 0}\subseteq L$, we have, by the definitions above: $\lim _{n\rightarrow \infty }l_n=1$ iff there exists $(c_n)_{n\geq 0}\subseteq L$ such that $(c_n)_{n\geq 0}\uparrow 1$ and, for all $n\in \N $, $c_n\leq d_L(l_n,1)$ iff there exists $(c_n)_{n\geq 0}\subseteq L$ such that $(c_n)_{n\geq 0}\uparrow 1$ and, for all $n\in \N $, $c_n\leq l_n$, because $d_L(l_n,1)=l_n$ by Lemma \ref{l2.2}, (\ref{l2.2(1)}) and (\ref{l2.2(2)}).
\end{remark}

\begin{lemma}{\rm \cite{ggap}} Let $(a_n)_{n\geq 0},(b_n)_{n\geq 0}\subseteq L$ and $a,b\in L$. If $\lim _{n\rightarrow \infty }a_n=a$ and $\lim _{n\rightarrow \infty }b_n=b$, then $\lim _{n\rightarrow \infty }(a_n\circ b_n)=a\circ b$ for each $\circ \in \{\vee ,\wedge ,\odot ,\rightarrow ,\leftrightarrow \}$. Thus $\lim _{n\rightarrow \infty }\neg \, a_n=\neg \, a$ and, if  $a_n\leq b_n$ for all $n\in \N $ (or for all $n\geq k\in \N $), then $a\leq b$.
\label{l6.1}
\end{lemma}

\begin{lemma}{\rm \cite{ggap}} Let $(a_n)_{n\geq 0}\subseteq L$ and $a\in L$. If $(a_n)_{n\geq 0}\uparrow a$ or $(a_n)_{n\geq 0}\downarrow a$ then $\lim _{n\rightarrow \infty }a_n=a$.
\label{l6.2}
\end{lemma}

A sequence $(a_n)_{n\geq 0}\subseteq X$ is said to be {\em $E$-convergent towards $a\in X$} iff $\lim _{n\rightarrow \infty }E(a_n,a)=1$; this is denoted by $a_n\stackrel{\textstyle E}{\textstyle \rightarrow }a$. $(a_n)_{n\geq 0}$ is said to be {\em $E$-Cauchy} iff $\lim _{n,m\rightarrow \infty }E(a_n,a_m)=1$.

\begin{lemma}
Assume that $E:X^2\rightarrow L$ is an $L$-equality and let $(a_n)_{n\geq 0}\subseteq X$, $a,a^{\prime }\in X$. If $a_n\stackrel{\textstyle E}{\textstyle \rightarrow }a$ and $a_n\stackrel{\textstyle E}{\textstyle \rightarrow }a^{\prime }$ then $a=a^{\prime }$.
\label{l6.3}
\end{lemma}
\begin{proof}
Assume that $a_n\stackrel{\textstyle E}{\textstyle \rightarrow }a$ and $a_n\stackrel{\textstyle E}{\textstyle \rightarrow }a^{\prime }$, that is $\lim _{n\rightarrow \infty }E(a_n,a)=1$ and $\lim _{n\rightarrow \infty }E(a_n,a^{\prime })=1$, thus, by Lemma \ref{l6.1}, $\lim _{n\rightarrow \infty }(E(a_n,a)\odot E(a_n,a^{\prime }))=1\odot 1=1$. But, for all $n\in \N $, $E(a_n,a)\odot E(a_n,a^{\prime })\leq E(a,a^{\prime })$, thus $E(a,a^{\prime })=1$ by Lemma \ref{l6.1}, so $a=a^{\prime }$.\end{proof}

\begin{lemma}
If $E$ is an $L$-equality, then any $E$-convergent sequence is $E$-Cauchy.
\label{l6.4}
\end{lemma}
\begin{proof}
Let $(a_n)_{n\geq 0}\subseteq X$ and $a\in X$ such that $a_n\stackrel{\textstyle E}{\textstyle \rightarrow }a$, that is $\lim _{n\rightarrow \infty }E(a_n,a)=1$. Then, by Lemma \ref{l6.1}, $\lim _{n,m\rightarrow \infty }E(a_n,a_m)\geq \lim _{n,m\rightarrow \infty }(E(a_n,a)\odot E(a_m,a))=1\odot 1=1$, therefore $(a_n)_{n\geq 0}$ is $E$-Cauchy.\end{proof}

Until mentioned otherwise, let $A$ and $L$ be two residuated lattices and $E:A^2\rightarrow L$ an $L$-similarity relation.

If $E:A^2\rightarrow L$ is an $L$-equality and any $E$-Cauchy sequence is $E$-convergent, then the residuated lattice $A$ is said to be {\em $E$-complete}.

For any function $s:A\rightarrow L$, we denote by $\rho _s:A^2\rightarrow L$ the function defined by: for all $a,b\in A$, $\rho _s(a,b)=s(d_A(a,b))$.

\begin{lemma}
Let $s:A\rightarrow L$ be an order-preserving type I state. Then, for all $a,b,x,y\in A$, we have:

\begin{enumerate}
\item\label{l6.5(1)} $\rho _s(a,b)\leq \rho _s(\neg \, a,\neg \, b)$;
\item\label{l6.5(2)} $\rho _s(a,b)\odot \rho _s(x,y)\leq \rho _s(a\circ x,b\circ y)$, for each $\circ \in \{\vee ,\wedge ,\odot ,\rightarrow ,\leftrightarrow \}$;
\item\label{l6.5(3)} $\rho _s(a,b)\leq d_L(s(a),s(b))$;
\item\label{l6.5(4)} if $a$ and $b$ are comparable, then: $\rho _s(a,b)=d_L(s(a),s(b))$;
\item\label{l6.5(5)} $\rho _s(a,x)\odot \rho _s(b,y)\leq d_L(\rho _s(a,b),\rho _s(x,y))$.
\end{enumerate}
\label{l6.5}
\end{lemma}

\begin{proof}
\noindent (\ref{l6.5(1)}) By Lemma \ref{lA}, (\ref{lA(4)}) and the fact that $s$ is order-preserving, $\rho _s(a,b)=s(d_A(a,b))\leq s(d_A(\neg \, a,\neg \, b))=\rho _s(\neg \, a,\neg \, b)$.

\noindent (\ref{l6.5(2)}) Let $\circ \in \{\vee ,\wedge ,\odot ,\rightarrow ,\leftrightarrow \}$. By Proposition \ref{p3.8}, (\ref{p3.8(1)}), Lemma \ref{lA}, (\ref{lA(5)}) and the fact that $s$ is order-preserving, $\rho _s(a,b)\odot \rho _s(x,y)=s(d_A(a,b))\odot s(d_A(x,y))\leq s(d_A(a,b)\odot d_A(x,y))\leq s(d_A(a\circ x,b\circ y))=\rho _s(a\circ x,b\circ y)$.

\noindent (\ref{l6.5(3)}) By Proposition \ref{p3.8}, (\ref{p3.8(5)}).

\noindent (\ref{l6.5(4)}) Assume, for instance, that $b\leq a$. Then, by Lemma \ref{l2.2}, (\ref{l2.2(3)}) and Proposition \ref{p3.3}, (\ref{p3.3(2)}), $\rho _s(a,b)=s(d_A(a,b))=s(a\rightarrow b)=s(a)\rightarrow s(b)=d_L(s(a),s(b))$. 

\noindent (\ref{l6.5(5)}) By Proposition \ref{p3.8}, (\ref{p3.8(6)}).\end{proof}

\begin{proposition}
If $s:A\rightarrow L$ is an order-preserving type I state, then $\rho _s$ is an $L$-similarity relation on $A$.
\label{p6.6}
\end{proposition}
\begin{proof}
By Proposition \ref{p3.8}, (\ref{p3.8(1)}), Lemma \ref{lA}, (\ref{lA(5)}) and the fact that $s$ is order-preserving, $\rho _s(a,b)\odot \rho _s(b,c)=s(d_A(a,b))\odot s(d_A(b,c))\leq s(d_A(a,b)\odot d_A(b,c))\leq s(d_A(a,c))=\rho _s(a,c)$.\end{proof}

If $s:A\rightarrow L$ is a generalized Bosbach state or a Rie\v can state, then we will say that $s$ is {\em faithful} iff, for all $a\in A$, $s(a)=1$ implies  $a=1$.

\begin{remark}
By Lemma \ref{lA}, (\ref{lA(1)}), if $s:A\rightarrow L$ is a faithful order-preserving type I state, then $\rho _s$ is an $L$-equality on $A$.

\label{r6.7}
\end{remark}

\begin{lemma}
Let $s:A\rightarrow L$ be a faithful order-preserving type I state, $(a_n)_{n\geq 0},(b_n)_{n\geq 0}\subseteq A$ and $a,b\in A$. If $a_n\stackrel{\textstyle \rho _s}{\textstyle \rightarrow }a$ and $b_n\stackrel{\textstyle \rho _s}{\textstyle \rightarrow }b$, then $a_n\circ b_n\stackrel{\textstyle \rho _s}{\textstyle \rightarrow }a\circ b$ for each $\circ \in \{\vee ,\wedge ,\odot ,\rightarrow ,\leftrightarrow \}$. From this and the definitions of $\neg \, $ and $\leq $, it follows that $\neg \, a_n\stackrel{\textstyle \rho _s}{\textstyle \rightarrow }\neg \, a$ and, if $a_n\leq b_n$ for all $n\in \N $ (or for all $n\geq k\in \N $), then $a\leq b$. 
\label{l6.8}
\end{lemma}
\begin{proof}
Apply Lemma \ref{l6.1} and Lemma \ref{l6.5}, (\ref{l6.5(2)}).\end{proof}

Let $s:A\rightarrow L$ be an arbitrary function and $a\in A$. Then $s$ is said to be:

\begin{itemize}

\item {\em $\uparrow $-continuous in $a$} iff, for any sequence $(a_n)_{n\geq 0}\subseteq A$ such that $a_n\uparrow a$, we have $\lim _{n\rightarrow \infty }s(a_n)=s(a)$;

\item {\em $\downarrow $-continuous in $a$} iff, for any sequence $(a_n)_{n\geq 0}\subseteq A$ such that $a_n\downarrow a$, we have $\lim _{n\rightarrow \infty }s(a_n)=s(a)$;

\item {\em continuous in $a$} iff it is $\uparrow $-continuous in $a$ and $\downarrow $-continuous in $a$. 
\end{itemize}

$s$ is said to be {\em $\uparrow $-continuous} (respectively {\em $\downarrow $-continuous}, {\em continuous}) iff it is $\uparrow $-continuous (respectively $\downarrow $-continuous, continuous) in any $a\in A$.

\begin{proposition}
Assume that $L$ is involutive and let $s:A\rightarrow L$ be a type I state and $a\in A$. If $s$ is $\downarrow $-continuous in $a$ then it is also $\uparrow $-continuous in $a$. Thus, if $s$ is $\downarrow $-continuous then it is also $\uparrow $-continuous.
\label{p6.9}
\end{proposition}
\begin{proof}
Assume that $s$ is $\downarrow $-continuous and let $(a_n)_{n\geq 0}\subseteq A$ such that $a_n\uparrow a$, that is $a_n\uparrow $ and $\bigvee _{n\in \N }a_n=a$. Then, by Lemma \ref{l2.3}, (\ref{l2.3(3)}) and Lemma \ref{l2.2}, (\ref{l2.2(9)}), $(\neg \, a_n)_{n\geq 0}\downarrow $ and $\bigwedge _{n\geq 0}(\neg \, a_n)=\neg \, (\bigvee _{n\geq 0}a_n)=\neg \, a$, thus $\neg \, a_n\downarrow \neg \, a$. By Lemma \ref{l6.1}, Proposition \ref{p3.6}, (\ref{p3.6(1)}) and the fact that $L$ is involutive, $\neg \, \lim _{n\rightarrow \infty }s(a_n)=\lim _{n\rightarrow \infty }\neg \, s(a_n)=\lim _{n\rightarrow \infty }s(\neg \, a_n)=s(\neg \, a)=\neg \, s(a)$, hence $\lim _{n\rightarrow \infty }s(a_n)=s(a)$. Therefore $s$ is $\uparrow $-continuous.\end{proof}

\begin{proposition}
Let $s:A\rightarrow L$ be a type II state and $a\in A$. If $s$ is $\downarrow $-continuous in $a$ then it is also $\uparrow $-continuous in $a$. Thus, if $s$ is $\downarrow $-continuous then it is also $\uparrow $-continuous.
\label{p6.10}
\end{proposition}
\begin{proof}
Assume that $s$ is $\downarrow $-continuous and let $(a_n)_{n\geq 0}\subseteq A$ such that $a_n\uparrow a$. Then, by the proof of Proposition \ref{p6.9}, $\neg \, a_n\downarrow \neg \, a$, hence $\lim _{n\rightarrow \infty }s(\neg \, a_n)=\neg \, a$. By Proposition \ref{p3.9}, (\ref{p3.9(3)}) and (\ref{p3.9(2)}), and Lemma \ref{l6.1}, $\lim _{n\rightarrow \infty }s(a_n)=\lim _{n\rightarrow \infty }s(\neg \, \neg \, a_n)=\lim _{n\rightarrow \infty }\neg \, s(\neg \, a_n)=\neg \, \lim _{n\rightarrow \infty }s(\neg \, a_n)=\neg \, s(\neg \, a)=s(a)$. Therefore $s$ is $\uparrow $-continuous.\end{proof}

\begin{proposition}
Let $A$ be an MV-algebra and $s:A\rightarrow L$ an order-preserving type I state. Let us consider the following statements:
\begin{enumerate}
\item\label{p6.12(1)} $s$ is $\uparrow $-continuous in $1$;
\item\label{p6.12(2)} $s$ is $\uparrow $-continuous;
\item\label{p6.12(3)} $s$ is $\downarrow $-continuous in $0$;
\item\label{p6.12(4)} $s$ is $\downarrow $-continuous;
\item\label{p6.12(5)} $s$ is continuous.
\end{enumerate}

Then (\ref{p6.12(2)})$\Leftrightarrow $(\ref{p6.12(1)})$\Rightarrow $(\ref{p6.12(4)})$\Rightarrow $(\ref{p6.12(3)}). If $L$ is involutive then (\ref{p6.12(1)}) iff (\ref{p6.12(2)}) iff (\ref{p6.12(3)}) iff (\ref{p6.12(4)}) iff (\ref{p6.12(5)}). 
\label{p6.12}
\end{proposition}
\begin{proof}
First let us prove that (\ref{p6.12(1)}) iff (\ref{p6.12(2)}). The converse implication is trivial. For the direct implication, let us assume that $s$ is $\uparrow $-continuous in $1$. Let $a\in A$ and $(a_n)_{n\geq 0}\subseteq A$ such that $a_n\uparrow a$, hence, for all $n\in \N $, $a_n\leq a$, which implies that $d_A(a_n,a)=a\rightarrow a_n$, by Lemma \ref{l2.2}, (\ref{l2.2(3)}). Thus, by Lemma \ref{l2.2}, (\ref{l2.2(4)}), $(d_A(a_n,a))_{n\geq 0}\uparrow $. Moreover, $\bigvee _{n\geq 0}d_A(a_n,a)=\bigvee _{n\geq 0}(a\rightarrow a_n)=a\rightarrow (\bigvee _{n\geq 0}a_n)=a\rightarrow a=1$, by Lemma \ref{mvlema}, (\ref{mvvee}), and Lemma \ref{l2.2}, (\ref{l2.2(3)}). Thus $(d_A(a_n,a))_{n\geq 0}\uparrow 1$. By Lemma \ref{l6.1}, the fact that, for all $n\in \N $, $a_n\leq a$, and Lemma \ref{l6.5}, (\ref{l6.5(4)}), it follows that $d_L(\lim _{n\rightarrow \infty }s(a_n),s(a))=\lim _{n\rightarrow \infty }d_L(s(a_n),s(a))=\lim _{n\rightarrow \infty }s(d_A(a_n,a))=s(1)=1$. Hence $\lim _{n\rightarrow \infty }s(a_n)=s(a)$, therefore $s$ is $\uparrow $-continuous in $a$.

Now let us prove that (\ref{p6.12(1)}) implies (\ref{p6.12(4)}). Thus let us assume that $s$ is $\uparrow $-continuous in $1$. Let $a\in A$ and $(a_n)_{n\geq 0}\subseteq A$ such that $a_n\downarrow a$, hence, for all $n\in \N $, $a_n\geq a$, which implies that $d_A(a_n,a)=a_n\rightarrow a$, by Lemma \ref{l2.2}, (\ref{l2.2(3)}). Thus, by Lemma \ref{l2.2}, (\ref{l2.2(4)}), $(d_A(a_n,a))_{n\geq 0}\uparrow $. Moreover, by Lemma \ref{mvlema}, (\ref{mvwedge}). $\bigvee _{n\geq 0}d_A(a_n,a)=\bigvee _{n\geq 0}(a_n\rightarrow a)=(\bigwedge _{n\geq 0}a_n)\rightarrow a=a\rightarrow a=1$, by Lemma \ref{mvlema}, (\ref{mvwedge}), and Lemma \ref{l2.2}, (\ref{l2.2(3)}). Thus $(d_A(a_n,a))_{n\geq 0}\uparrow 1$. By Lemma \ref{l6.1}, the fact that, for all $n\in \N $, $a_n\geq a$, and Lemma \ref{l6.5}, (\ref{l6.5(4)}), it follows that $d_L(\lim _{n\rightarrow \infty }s(a_n),s(a))=\lim _{n\rightarrow \infty }d_L(s(a_n),s(a))=\lim _{n\rightarrow \infty }s(d_A(a_n,a))=s(1)=1$. Hence $\lim _{n\rightarrow \infty }s(a_n)=s(a)$, therefore $s$ is $\downarrow $-continuous in $a$.

Trivially (\ref{p6.12(4)}) implies (\ref{p6.12(3)}).

Now let us assume that $L$ is involutive. For proving the equivalences in the enunciation it remains to show that (\ref{p6.12(3)}) implies (\ref{p6.12(1)}). Thus, let us assume that $s$ is $\downarrow $-continuous in $0$ and let $(a_n)_{n\geq 0}\subseteq A$ such that $a_n\uparrow 1$. Then $a_n\uparrow $, thus $\neg \, a_n\downarrow $, by Lemma \ref{l2.2}, (\ref{l2.2(4)}). Moreover, by Lemma \ref{l2.2}, (\ref{l2.2(9)}) and Lemma \ref{l2.3}, (\ref{0neg1}), $\bigwedge _{n\geq 0}\neg \, a_n=\neg \, (\bigvee _{n\geq 0}a_n)=\neg \, 1=0$. So $\neg \, a_n\downarrow 0$, hence $\lim _{n\rightarrow \infty }s(\neg \, a_n)=s(0)=0$. By Lemma \ref{p3.6}, (\ref{p3.6(1)}), and Lemma \ref{l6.1}, $\lim _{n\rightarrow \infty }s(a_n)=\lim _{n\rightarrow \infty }\neg \, \neg \, s(a_n)=\lim _{n\rightarrow \infty }\neg \, s(\neg \, a_n)=\neg \, (\lim _{n\rightarrow \infty }s(\neg \, a_n))=\neg \, 0=1=s(1)$. Hence $s$ is $\uparrow $-continuous in $1$.\end{proof}

Let $E:A^2\rightarrow L$ be an $L$-similarity relation and $s:A\rightarrow L$ an arbitrary function. We say that $s$ is {\em $E$-continuous in $a\in A$} iff, for all $(a_n)_{n\geq 0}\subseteq A$ such that $a_n\stackrel{\textstyle E}{\textstyle \rightarrow }a$, we have $\lim _{n\rightarrow \infty }s(a_n)=s(a)$. We say that $s$ is {\em $E$-continuous} iff it is $E$-continuous in any $a\in A$. Actually, these definitions are valid for the residuated lattice $A$ replaced by an arbitrary nonempty set $X$, but we shall not work with them in this general case.

\begin{proposition}
Any order-preserving type I state $s:A\rightarrow L$ is $\rho _s$-continuous.
\label{p6.11}
\end{proposition}
\begin{proof}
Let $s:A\rightarrow L$ be an order-preserving type I state, $a\in A$ and $(a_n)_{n\geq 0}\subseteq A$ such that $a_n\stackrel{\textstyle \rho _s}{\textstyle \rightarrow }a$, that is $\lim _{n\rightarrow \infty }\rho _s(a_n,a)=1$. By Lemma \ref{l6.5}, (\ref{l6.5(3)}), for all $n\in \N $, $\rho _s(a_n,a)\leq d_L(s(a_n),s(a))$. By Lemma \ref{l6.1} and Lemma \ref{lA}, (\ref{lA(1)}), it follows that $1=\lim _{n\rightarrow \infty }d_L(s(a_n),s(a))=d_L(\lim _{n\rightarrow \infty }s(a_n),s(a))$, hence $\lim _{n\rightarrow \infty }s(a_n)=s(a)$, thus $s$ is $\rho _s$-continuous in $a$.\end{proof}

A residuated lattice $A$ is said to be {\em $\sigma $-complete} iff any sequence $(a_n)_{n\geq 0}\subseteq A$ has a supremum and an infimum in $A$. Notice that: $A$ is {\em $\sigma $-complete} iff any increasing sequence in $A$ has a supremum in $A$ and any decreasing sequence in $A$ has an infimum in $A$. This is easily shown, because, if the latter is verified, then, for any $(a_n)_{n\geq 0}\subseteq A$, if we consider the increasing sequence $(\bigvee _{k=0}^{n}a_k)_{n\geq 0}$, that has a supremum by the hypothesis, and the decreasing sequence $(\bigwedge _{k=0}^{n}a_k)_{n\geq 0}$, that has an infimum by the hypothesis, then $\bigvee _{n\geq 0}(\bigvee _{k=0}^{n}a_k)=\bigvee _{n\geq 0}a_n$ and $\bigwedge _{n\geq 0}(\bigwedge _{k=0}^{n}a_k)=\bigwedge _{n\geq 0}a_n$, which can easily be shown by the definition of the supremum and that of the infimum.

\begin{proposition}
Let $s:A\rightarrow L$ be a faithful order-preserving type I state, $A$ be $\rho _s$-complete and $L$ be $\sigma $-complete. Then $A$ is $\sigma $-complete and $s$ is $\uparrow $-continuous in $1$.
\label{p6.14}
\end{proposition}
\begin{proof}
By Remark \ref{r6.7}, $\rho _s$ is an $L$-equality on $A$. Let $(a_n)_{n\geq 0}\subseteq A$ be such that $(a_n)_{n\geq 0}\uparrow $. Since $s$ is order-preserving, it follows that $(s(a_n))_{n\geq 0}\uparrow $ in $L$. Since $L$ is $\sigma $-complete, there exists $\bigvee _{n\geq 0}s(a_n)$ in $L$, thus $(s(a_n))_{n\geq 0}\uparrow \bigvee _{n\geq 0}s(a_n)$, therefore, by Lemma \ref{l6.2}, $(s(a_n))_{n\geq 0}$ is convergent in $L$, hence $(s(a_n))_{n\geq 0}$ is Cauchy. By Lemma \ref{l6.5}, (\ref{l6.5(4)}), for all $n,m\in \N $, $\rho _s(a_n,a_m)=d_L(s(a_n),s(a_m))$, thus $\lim _{n,m\rightarrow \infty }\rho _s(a_n,a_m)=\lim _{n,m\rightarrow \infty }d_L(s(a_n),s(a_m))=1$, so $(a_n)_{n\geq 0}$ is $\rho _s$-Cauchy. But $A$ is $\rho _s$-complete, therefore there exists $a\in A$ such that $a_n\stackrel{\textstyle \rho _s}{\textstyle \rightarrow }a$. Let $k\in \N $, arbitrary but fixed. By Lemma \ref{l6.8}, $a_n\vee a_k\stackrel{\textstyle \rho _s}{\textstyle \rightarrow }a\vee a_k$. Since $(a_n)_{n\geq 0}\uparrow $, we have that, for all $n\geq k$, $a_n\vee a_k=a_n$, and, since $a_n\stackrel{\textstyle \rho _s}{\textstyle \rightarrow }a$, we may conclude that $(a_n\vee a_k)_{n\geq 0}\stackrel{\textstyle \rho _s}{\textstyle \rightarrow }a$. By Lemma \ref{l6.3}, it follows that $a\vee a_k=a$, that is $a_k\leq a$. Thus $a_n\leq a$ for all $n\in \N $. Now let $b\in A$ such that, for all $n\in \N $, $a_n\leq b$, that is $a_n\vee b=b$. By Lemma \ref{l6.8}, it follows that $a_n\vee b\stackrel{\textstyle \rho _s}{\textstyle \rightarrow }a\vee b$, that is $b\stackrel{\textstyle \rho _s}{\textstyle \rightarrow }a\vee b$, that is $b=a\vee b$, thus $a\leq b$. Hence $\bigvee _{n\geq 0}a_n=a$. Analogously one can prove that any decreasing sequence in $A$ has an infimum in $A$. Therefore $A$ is $\sigma $-complete.

It remains to show that $s$ is $\uparrow $-continuous in $1$. Let $(a_n)_{n\geq 0}\subseteq A$ such that $a_n\uparrow 1$. By the above, there exists $a\in A$ such that $a_n\stackrel{\textstyle \rho _s}{\textstyle \rightarrow }a$ and $a_n\leq a$ for all $n\in \N $, thus $1=\bigvee _{n\geq 0}a_n\leq a$, hence $a=1$. So $a_n\stackrel{\textstyle \rho _s}{\textstyle \rightarrow }1$, that is $\lim _{n\rightarrow \infty }\rho _s(a_n,1)=1$. But, for all $n\in \N $, $\rho _s(a_n,1)=s(d_A(a_n,1))=s(a_n)$, as Lemma \ref{l2.2}, (\ref{l2.2(1)}) and (\ref{l2.2(2)}), shows. So $\lim _{n\rightarrow \infty }s(a_n)=1=s(1)$, hence $s$ is $\uparrow $-continuous in $1$.\end{proof}

\begin{remark}
Let $A$ be an MV-algebra, $L$ a $\sigma $-complete involutive residuated lattice and $s:A\rightarrow L$ a faithful order-preserving type I state such that $A$ is $\rho _s$-complete. Then, by Propositions \ref{p6.12} and \ref{p6.14}, $s$ is continuous. This way, Theorem 3.7 from \cite{ioanal} becomes a particular case of Proposition \ref{p6.14}.
\label{r6.15}
\end{remark}

In \cite{ioanal}, the author defines and studies the {\em metric completion} of an MV-algebra endowed with an MV-state. This is a version for MV-algebras of the metric completion of an $l$-group with a state (see \cite{good}). In the following, we shall analyse the way in which this construction can be generalized to the case of a residuated lattice $A$ endowed with an order-preserving type I state.

Throughout the rest of this section, $A$ and $L$ will be two residuated lattices such that $L$ is Cauchy-complete and $s:A\rightarrow L$ will be an order-preserving type I state. 

By Proposition \ref{p6.6}, $\rho _s$ is an $L$-similarity relation on $A$. Let us denote by ${\cal C}_s(A)$ the set of the $\rho _s$-Cauchy sequences in $A$ and let us define on ${\cal C}_s(A)$ the following binary operations: for all $\circ \in \{\vee ,\wedge ,\odot ,\rightarrow ,\leftrightarrow \}$, we define: for all $\underline{a}=(a_n)_{n\geq 0},\underline{b}=(b_n)_{n\geq 0}\in {\cal C}_s(A)$, $\underline{a}\circ \underline{b}=(a_n\circ b_n)_{n\geq 0}\in {\cal C}_s(A)$, because, by Lemma \ref{l6.5}, (\ref{l6.5(2)}) and Lemma \ref{l6.1}, $\lim _{n,m\rightarrow \infty }\rho _s(a_n\circ b_n,a_m\circ b_m)\geq (\lim _{n,m\rightarrow \infty }\rho _s(a_n,a_m))\odot (\lim _{n,m\rightarrow \infty }\rho _s(b_n,b_m))=1$, thus $\lim _{n,m\rightarrow \infty }\rho _s(a_n\circ b_n,a_m\circ b_m)=1$, so $(a_n\circ b_n)_{n\geq 0}$ is a $\rho _s$-Cauchy sequence in $A$. We denote $\underline{0}=(0)_{n\geq 0},\underline{1}=(1)_{n\geq 0}\in {\cal C}_s(A)$, as all constant sequences in $A$ are obviously $\rho _s$-Cauchy (see Lemma \ref{lA}, (\ref{lA(1)})). It is immediate that $({\cal C}_s(A),\vee ,\wedge ,\odot ,\rightarrow ,\underline{0},\underline{1})$ is a residuated lattice, whose biresiduum is $\leftrightarrow $ and whose negation is: for all $\underline{a}=(a_n)_{n\geq 0}\in {\cal C}_s(A)$, $\neg \, \underline{a}=\underline{a}\rightarrow \underline{0}=(a_n\rightarrow 0)_{n\geq 0}=(\neg \, a_n)_{n\geq 0}\in {\cal C}_s(A)$, as $\underline{a}\rightarrow \underline{0}\in {\cal C}_s(A)$.

Let $\underline{a}=(a_n)_{n\geq 0},\underline{b}=(b_n)_{n\geq 0}\in {\cal C}_s(A)$. By Lemma \ref{l6.5}, (\ref{l6.5(5)}) and Lemma \ref{l6.1}, for all $n,m\in \N $, $\rho _s(a_n,a_m)\odot \rho _s(b_n,b_m)\leq d_L(\rho _s(a_n,b_n),\rho _s(a_m,b_m))$, hence $\lim _{n,m\rightarrow \infty }d_L(\rho _s(a_n,b_n),\rho _s(a_m,b_m))=1$, thus the sequence $(\rho _s(a_n,b_n))_{n\geq 0}\subseteq L$ is Cauchy and hence convergent, since $L$ is Cauchy-complete.

Let us define on ${\cal C}_s(A)$ the following binary relation: $\sim \subseteq {\cal C}_s(A)\times {\cal C}_s(A)$, defined by: for all $\underline{a}=(a_n)_{n\geq 0},\underline{b}=(b_n)_{n\geq 0}\in {\cal C}_s(A)$, $\underline{a}\sim \underline{b}$ iff $\lim _{n\rightarrow \infty }\rho _s(a_n,b_n)=1$. $\rho _s$ is an $L$-similarity relation on $A$, hence, by applying Lemma \ref{l6.1}, we obtain that $\sim $ is an equivalence relation on ${\cal C}_s(A)$. Let us consider the quotient set $\tilde{A}_s:={\cal C}_s(A)/_{\sim }=\{\tilde{\underline{a}}|\underline{a}\in {\cal C}_s(A)\}$, where we denoted by $\tilde{\underline{a}}$ the equivalence class of a sequence $\underline{a}\in {\cal C}_s(A)$ with respect to $\sim $. Let us define on $\tilde{A}_s$ the following binary operations: for all $\circ \in \{\vee ,\wedge ,\odot ,\rightarrow ,\leftrightarrow \}$, we define: for all $\underline{a},\underline{b}\in {\cal C}_s(A)$, $\tilde{\underline{a}}\circ \tilde{\underline{b}}=\widetilde{\underline{a}\circ \underline{b}}\in \tilde{A}_s$. Let us prove that all of these operations are well defined. Let $\circ \in \{\vee ,\wedge ,\odot ,\rightarrow ,\leftrightarrow \}$ and let $\underline{a}=(a_n)_{n\geq 0},\underline{a^{\prime }}=(a^{\prime }_n)_{n\geq 0},\underline{b}=(b_n)_{n\geq 0},\underline{b^{\prime }}=(b^{\prime }_n)_{n\geq 0}\in {\cal C}_s(A)$ such that $\underline{a}\sim \underline{a^{\prime }}$ and $\underline{b}\sim \underline{b^{\prime }}$, that is: $\lim _{n\rightarrow \infty }\rho _s(a_n,a^{\prime }_n)=\lim _{n\rightarrow \infty }\rho _s(b_n,b^{\prime }_n)=1$. By Lemma \ref{l6.5}, (\ref{l6.5(2)}) and Lemma \ref{l6.1}, it follows that $\lim _{n\rightarrow \infty }\rho _s(a_n\circ b_n,a^{\prime }_n\circ b^{\prime }_n)=1$, that is $\underline{a}\circ \underline{b}\sim \underline{a^{\prime }}\circ \underline{b^{\prime }}$, that is $\widetilde{\underline{a}\circ \underline{b}}=\widetilde{\underline{a^{\prime }}\circ \underline{b^{\prime }}}$. So $\circ $ is well defined. Thus $\sim $ has become a congruence relation on the residuated lattice $({\cal C}_s(A),\vee ,\wedge ,\odot ,\rightarrow ,\underline{0},\underline{1})$, and the fact that residuated lattices form an equational class ensures us that $(\tilde{A}_s,\vee ,\wedge ,\odot ,\rightarrow ,\tilde{\underline{0}},\tilde{\underline{1}})$ is a residuated lattice, whose biresiduum is obviously $\leftrightarrow $ and whose negation is: for all $\underline{a}\in {\cal C}_s(A)$, $\neg \, \tilde{\underline{a}}=\tilde{\underline{a}}\rightarrow \tilde{\underline{0}}=\widetilde{\underline{a}\rightarrow \underline{0}}=\widetilde{\neg \, \underline{a}}\in \tilde{A}_s$. 

\begin{lemma}
If $L$ is involutive then $\tilde{A}_s$ is involutive.
\label{l6.16}
\end{lemma}
\begin{proof}
By Lemma \ref{l2.3}, (\ref{l2.3(2)}), Proposition \ref{p3.3}, (\ref{p3.3(2)}), Proposition \ref{p3.6}, (\ref{p3.6(1)}), the fact that $L$ is involutive and Lemma \ref{l2.2}, (\ref{l2.2(3)}), for all $a\in A$, $s(\neg \, \neg \, a\rightarrow a)=s(\neg \, \neg \, a)\rightarrow s(a)=\neg \, \neg \, s(a)\rightarrow s(a)=s(a)\rightarrow s(a)=1$ and thus $\rho _s(a,\neg \, \neg \, a)=s(d_A(a,\neg \, \neg \, a))=s(\neg \, \neg \, a\rightarrow a)=1$. Thus, for all $a\in A$, $\rho _s(a,\neg \, \neg \, a)=s(d_A(a,\neg \, \neg \, a))=s()$. Let $\underline{a}=(a_n)_{n\geq 0}\in {\cal C}_s(A)$ and let us consider the sequence $\neg \, \neg \, \underline{a}=(\neg \, \neg \, a_n)_{n\geq 0}\in {\cal C}_s(A)$. For all $n\in \N $, $\rho _s(a_n,\neg \, \neg \, a_n)=1$, hence $\neg \, \neg \, \underline{a}\sim \underline{a}$, that is $\widetilde{\neg \, \neg \, \underline{a}}=\tilde{\underline{a}}$, that is $\neg \, \neg \, \tilde{\underline{a}}=\tilde{\underline{a}}$.\end{proof}

For all $a\in A$, let us denote in this paragraph the constant sequence $\underline{a}=(a)_{n\geq 0}\in {\cal C}_s(A)$. The function $\psi _s:A\rightarrow {\cal C}_s(A)$, defined by $\psi _s(a)=\underline{a}$ for all $a\in A$, is obviously an injective residuated lattice morphism. By composing the canonical projection from ${\cal C}_s(A)$ to the quotient residuated lattice $\tilde{A}_s$ with the morphism $\psi _s$, we obtain the residuated lattice morphism $\varphi _s:A\rightarrow \tilde{A}_s$, defined by $\varphi _s(a)=\tilde{\underline{a}}$ for all $a\in A$.

\begin{lemma}
Let $\underline{a}=(a_n)_{n\geq 0},\underline{b}=(b_n)_{n\geq 0},\underline{c}=(c_n)_{n\geq 0},\underline{d}=(d_n)_{n\geq 0}\in {\cal C}_s(A)$. If $\underline{a}\sim \underline{c}$ and $\underline{b}\sim \underline{d}$, then $\lim _{n\rightarrow \infty }\rho _s(a_n,b_n)=\lim _{n\rightarrow \infty }\rho _s(c_n,d_n)$.
\label{l6.17}
\end{lemma}
\begin{proof}
By the fact that $\rho _s$ is an $L$-similarity relation on $A$ and Lemma \ref{l2.2}, (\ref{l2.2nenum2}), we have that: for all $n\in \N $, $\rho _s(c_n,a_n)\odot \rho _s(a_n,b_n)\odot \rho _s(b_n,d_n)\leq \rho _s(c_n,d_n)$. By Lemma \ref{lA}, (\ref{lA(2)}) and Lemma \ref{l6.1}, it follows that $1\odot (\lim _{n\rightarrow \infty }\rho _s(a_n,b_n))\odot 1\leq \lim _{n\rightarrow \infty }\rho _s(b_n,d_n)$, hence $\lim _{n\rightarrow \infty }\rho _s(a_n,b_n)\leq \lim _{n\rightarrow \infty }\rho _s(c_n,d_n)$. The converse inequality results in a similar way.\end{proof}

By Lemma \ref{l6.17}, we can define the function $\tilde{\rho _s}:\tilde{A}_s\times \tilde{A}_s\rightarrow L$, by: for all $\underline{a}=(a_n)_{n\geq 0},\underline{b}=(b_n)_{n\geq 0}\in {\cal C}_s(A)$, $\tilde{\rho _s}(\tilde{\underline{a}},\tilde{\underline{b}})=\lim _{n\rightarrow \infty }\rho _s(a_n,b_n)$.

\begin{proposition}
$\tilde{\rho _s}$ is an $L$-similarity relation on $\tilde{A}_s$.
\label{p6.18}
\end{proposition}
\begin{proof}
It is immediate that $\tilde{\rho _s}$ is reflexive and symmetric. In order to prove that it is transitive, let us consider $\underline{a}=(a_n)_{n\geq 0},\underline{b}=(b_n)_{n\geq 0},\underline{c}=(c_n)_{n\geq 0}\in {\cal C}_s(A)$. Then, by the fact that $\rho _s$ is an $L$-similarity relation on $A$, it follows that, for all $n\in \N $, $\rho _s(a_n,b_n)\odot \rho _s(b_n,c_n)\leq \rho _s(a_n,c_n)$. By applying Lemma \ref{l6.1}, we obtain: $\tilde{\rho _s}(\tilde{\underline{a}},\tilde{\underline{b}})\odot \tilde{\rho _s}(\tilde{\underline{b}},\tilde{\underline{c}})=\lim _{n\rightarrow \infty }(\rho _s(a_n,b_n)\odot \rho _s(b_n,c_n))\leq \tilde{\rho _s}(\tilde{\underline{a}},\tilde{\underline{c}})$.\end{proof}

\begin{lemma}
Let $\underline{a}=(a_n)_{n\geq 0},\underline{b}=(b_n)_{n\geq 0}\in {\cal C}_s(A)$. If $\underline{a}\sim \underline{b}$, then $\lim _{n\rightarrow \infty }s(a_n)=\lim _{n\rightarrow \infty }s(b_n)$.
\label{l6.19}
\end{lemma}
\begin{proof}
By Lemma \ref{l6.5}, (\ref{l6.5(3)}), for all $n\in \N $, $\rho _s(a_n,b_n)\leq d_L(s(a_n),s(b_n))$. By Lemma \ref{l6.1} and the fact that $\lim _{n\rightarrow \infty }\rho _s(a_n,b_n)=1$, we have: $d_L(\lim _{n\rightarrow \infty }s(a_n),\lim _{n\rightarrow \infty }s(b_n))=\lim _{n\rightarrow \infty }d_L(s(a_n),s(b_n))=1$. By Lemma \ref{lA}, (\ref{lA(1)}), we get: $\lim _{n\rightarrow \infty }s(a_n)=\lim _{n\rightarrow \infty }s(b_n)$.\end{proof}

Lemma \ref{l6.19} allows us to define the function $\tilde{s}:\tilde{A}_s\rightarrow L$, for all $\underline{a}=(a_n)_{n\geq 0}\in {\cal C}_s(A)$, $\tilde{s}(\tilde{\underline{a}})=\lim _{n\rightarrow \infty }s(a_n)$.

\begin{proposition}
$\tilde{s}$ is a faithful order-preserving type I state.
\label{p6.20}
\end{proposition}
\begin{proof}
Obviously, $\tilde{s}(\tilde{\underline{0}})=0$ and $\tilde{s}(\tilde{\underline{1}})=1$. By Lemma \ref{l6.1}, $\tilde{s}$ is an order-preserving function.

Now let $\underline{a}=(a_n)_{n\geq 0},(b_n)_{n\geq 0}\in {\cal C}_s(A)$. Then, by Proposition \ref{p3.3}, (\ref{p3.3(3)}) and Lemma \ref{l6.1}, $\tilde{s}(\tilde{\underline{a}}\rightarrow \tilde{\underline{b}})=\lim _{n\rightarrow \infty }s(a_n\rightarrow b_n)=\lim _{n\rightarrow \infty }(s(a_n)\rightarrow s(a_n\wedge b_n))=(\lim _{n\rightarrow \infty }s(a_n))\rightarrow (\lim _{n\rightarrow \infty }s(a_n\wedge b_n))=\tilde{s}(\tilde{\underline{a}})\rightarrow \tilde{s}(\widetilde{\underline{a}\wedge \underline{b}})=\tilde{s}(\tilde{\underline{a}})\rightarrow \tilde{s}(\tilde{\underline{a}}\wedge \tilde{\underline{b}})$. Thus, by Proposition \ref{p3.3}, (\ref{p3.3(3)}), $\tilde{s}$ is a type I state. If $\tilde{s}(\tilde{\underline{a}})=1$, then, by Lemma \ref{l2.2}, (\ref{l2.2(1)}) and (\ref{l2.2(2)}), $\lim _{n\rightarrow \infty }\rho _s(a_n,1)=\lim _{n\rightarrow \infty }s(d_A(a_n,1))=\lim _{n\rightarrow \infty }s(a_n)=1$, so $\underline{a}\sim \underline{1}$, that is $\tilde{\underline{a}}=\tilde{\underline{1}}$. Hence $\tilde{s}$ is faithful.\end{proof}

The following theorem collects the main properties of $\tilde{A}_s$, $\tilde{\rho _s}$ and $\tilde{s}$.

\begin{theorem}
Let $A$ and $L$ be two residuated lattices, such that $L$ is Cauchy-complete, and $s:A\rightarrow L$ an order-preserving type I state. Then:

\begin{enumerate}
\item\label{t6.2(1)} $\tilde{A}_s$ is a residuated lattice; if $L$ is involutive then $\tilde{A}_s$ is also involutive;
\item\label{t6.2(2)} $\tilde{s}$ is a faithful order-preserving type I state;
\item\label{t6.2(4)} $\varphi _s$ is a residuated lattice morphism and $\tilde{s}\circ \varphi _s=s$;
\item\label{t6.2(5)} $\varphi _s$ is injective iff $s$ is faithful;
\item\label{t6.2(6)} $\tilde{\rho _s}=\rho _{\tilde{s}}$;
\item\label{t6.2(7)} for any $(a_n)_{n\geq 0}\subseteq A$ and $a\in A$, if $a_n\stackrel{\textstyle \rho _s}{\textstyle \rightarrow }a$, then $\varphi _s(a_n)\stackrel{\textstyle \tilde{\rho _s}}{\textstyle \rightarrow }\varphi _s(a)$;
\item\label{t6.2(8)} for any residuated lattice $C$, any faithful order-preserving type I state $m:C\rightarrow L$ such that $C$ is $\rho _m$-complete, and any residuated lattice morphism $f:A\rightarrow C$ such that $m\circ f=s$, there exists a residuated lattice morphism $\tilde{f}:\tilde{A}_s\rightarrow C$ such that $m\circ \tilde{f}=\tilde{s}$ and $\tilde{f}\circ \varphi _s=f$.
\end{enumerate}
\label{t6.2}
\end{theorem}
\begin{proof}
\noindent (\ref{t6.2(1)}) This is Lemma \ref{l6.16}.

\noindent (\ref{t6.2(2)}) This is Proposition \ref{p6.20}.

\noindent (\ref{t6.2(4)}) We know that $\varphi _s$ is a residuated lattice morphism. Let $a\in A$ and $\underline{a}=(a)_{n\geq 0}$. $(\tilde{s}\circ \varphi _s)(a)=\tilde{s}(\tilde{\underline{a}})=\lim _{n\rightarrow \infty }s(a)=s(a)$. Thus $\tilde{s}\circ \varphi _s=s$.

\noindent (\ref{t6.2(5)}) Let $a\in A$ and $\underline{a}=(a)_{n\geq 0}\in {\cal C}_s(A)$. We have the equivalences: $a\in {\rm Ker}(\varphi _s)$ iff $\varphi _s(a)=\tilde{\underline{1}}$ iff $\tilde{\underline{a}}=\tilde{\underline{1}}$ iff $\lim _{n\rightarrow \infty }\rho _s(a,1)=1$ iff $\lim _{n\rightarrow \infty }s(a)=1$ iff $s(a)=1$, by Lemma \ref{l2.2}, (\ref{l2.2(1)}) and (\ref{l2.2(2)}). Hence: $\varphi _s$ is injective iff ${\rm Ker}(\varphi _s)=\{1\}$ iff the fact that $s(a)=1$ implies $a=1$ iff $s$ is faithful. 

\noindent (\ref{t6.2(6)}) $\tilde{\rho _s},\rho _{\tilde{s}}:\tilde{A}_s\times \tilde{A}_s\rightarrow L$. For all $\underline{a}=(a_n)_{n\geq 0},\underline{b}=(b_n)_{n\geq 0}\in {\cal C}_s(A)$, we have the following equalities: $\rho _{\tilde{s}}(\tilde{\underline{a}},\tilde{\underline{b}})=\tilde{s}(d_{\tilde{A}_s}(\tilde{\underline{a}},\tilde{\underline{b}}))=\tilde{s}((d_A(a_n,b_n))_{n\geq 0})=\lim _{n\rightarrow \infty }s(d_A(a_n,b_n))=\lim _{n\rightarrow \infty }\rho _s(a_n,b_n)=\tilde{\rho _s}(\tilde{\underline{a}},\tilde{\underline{b}})$. Thus $\tilde{\rho _s}=\rho _{\tilde{s}}$.

\noindent (\ref{t6.2(7)}) Let $\underline{x}=(a_n)_{n\geq 0}\subseteq A$ and $a\in A$, such that $a_n\stackrel{\textstyle \rho _s}{\textstyle \rightarrow }a$, that is $\lim _{n\rightarrow \infty }\rho _s(a_n,a)=1$. Let us denote $\underline{a}=(a)_{n\geq 0}$. For all $n\in \N $, $\tilde{\rho _s}(\varphi _s(a_n),\varphi _s(a))=\rho _{\tilde{s}}(\varphi _s(a_n),\varphi _s(a))=\tilde{s}(d_{\tilde{A}_s}(\varphi _s(a_n),\varphi _s(a)))=\tilde{s}(d_{\tilde{A}_s}(\tilde{\underline{x}},\tilde{\underline{a}}))=\tilde{s}((d_A(a_n,a))_{n\geq 0})=\lim _{n\rightarrow \infty }s(d_A(a_n,a))=\lim _{n\rightarrow \infty }\rho _s(a_n,a)=1$. Hence $\lim _{n\rightarrow \infty }\tilde{\rho _s}(\varphi _s(a_n),\varphi _s(a))=1$, that is $\varphi _s(a_n)\stackrel{\textstyle \tilde{\rho _s}}{\textstyle \rightarrow }\varphi _s(a)$.

\noindent (\ref{t6.2(8)}) Let $C$, $m$ and $f$ be like in the enunciation. Then, by Remark \ref{r6.7}, $\rho _m$ is an $L$-equality on $C$. We shall denote by $\approx $ the congruence on ${\cal C}_m(C)$ defined in the same way as $\sim $ on ${\cal C}_s(A)$.

Let $(a_n)_{n\geq 0}\in {\cal C}_s(A)$, arbitrary but fixed, so $\lim _{n,k\rightarrow \infty }\rho _s(a_n,a_k)=1$. For all $n,k\in \N $, since $f$ is a residuated lattice morphism, we have: $\rho _m(f(a_n),f(a_k))=m(d_C(f(a_n),f(a_k)))=m(f(d_A(a_n,a_k)))=s(d_A(a_n,a_k))=\rho _s(a_n,a_k)$. Thus $\lim _{n,k\rightarrow \infty }\rho _m(f(a_n),f(a_k))=1$, that is $(f(a_n))_{n\geq 0}\in {\cal C}_m(C)$, so, since $C$ is $\rho _m$-complete, there exists $c\in C$ such that $f(a_n)\stackrel{\textstyle \rho _m}{\textstyle \rightarrow }c$. This element $c$ of $C$ is unique, as Lemma \ref{l6.3} shows. We set $\tilde{f}(\widetilde{(a_n)_{n\geq 0}})=c$.

Let us prove that $\tilde{f}$ is well defined. Let $(a_n)_{n\geq 0},(b_n)_{n\geq 0}\in {\cal C}_s(A)$, such that $(a_n)_{n\geq 0}\sim (b_n)_{n\geq 0}$. By the above, there exist $c,d\in C$ such that $f(a_n)\stackrel{\textstyle \rho _m}{\textstyle \rightarrow }c$ and $f(b_n)\stackrel{\textstyle \rho _m}{\textstyle \rightarrow }d$. We have to prove that $c=d$. The fact that $f(a_n)\stackrel{\textstyle \rho _m}{\textstyle \rightarrow }c$ is equivalent to $\lim _{n\rightarrow \infty }\rho _m(f(a_n),c)=1$, that is $(f(a_n))_{n\geq 0}\approx (c)_{n\geq 0}$ (the constant sequence). Analogously, $(f(b_n))_{n\geq 0}\approx (d)_{n\geq 0}$. By Lemma \ref{l6.19}, $\lim _{n\rightarrow \infty }s(a_n)=\lim _{n\rightarrow \infty }s(b_n)$, that is $\lim _{n\rightarrow \infty }m(f(a_n))=\lim _{n\rightarrow \infty }m(f(b_n))$. By the fact that $f$ is a residuated lattice morphism and $(a_n)_{n\geq 0}\sim (b_n)_{n\geq 0}$, it follows that $\lim _{n\rightarrow \infty }\rho _m(f(a_n),f(b_n))=\lim _{n\rightarrow \infty }m(d_C(f(a_n),f(b_n)))=\lim _{n\rightarrow \infty }m(f(d_A(a_n,b_n)))=\lim _{n\rightarrow \infty }s(d_A(a_n,b_n))=\lim _{n\rightarrow \infty }\rho _s(a_n,b_n)=1$, hence $(f(a_n))_{n\geq 0}\approx (f(b_n))_{n\geq 0}$. By the symmetry and the transitivity of $\approx $, it follows that $(c)_{n\geq 0}\approx (d)_{n\geq 0}$, thus $1=\lim _{n\rightarrow \infty }\rho _m(c,d)=\rho _m(c,d)=m(d_C(c,d))$. By the fact that $m$ is faithful and by Lemma \ref{lA}, (\ref{lA(1)}), it results that $c=d$, therefore $\tilde{f}$ is well defined.

Let us prove that $\tilde{f}$ defined this way is a residuated lattice morphism. It is trivial that $\tilde{f}(\tilde{\underline{0}})=0$ and $\tilde{f}(\tilde{\underline{1}})=1$. Now let $\circ \in \{\vee ,\wedge ,\odot ,\rightarrow \}$ and $(a_n)_{n\geq 0},(b_n)_{n\geq 0}\in {\cal C}_s(A)$. By the above, there exist $c,d\in C$ such that $f(a_n)\stackrel{\textstyle \rho _m}{\textstyle \rightarrow }c$ and $f(b_n)\stackrel{\textstyle \rho _m}{\textstyle \rightarrow }d$, and $\tilde{f}(\widetilde{(a_n)_{n\geq 0}})=c$ and $\tilde{f}(\widetilde{(b_n)_{n\geq 0}})=d$. By Lemma \ref{l6.8} and the fact that $f$ is a residuated lattice morphism, we have: $f(a_n\circ b_n)=f(a_n)\circ f(b_n)\stackrel{\textstyle \rho _m}{\textstyle \rightarrow }c\circ d$, thus $\tilde{f}(\widetilde{(a_n)_{n\geq 0}}\circ \widetilde{(b_n)_{n\geq 0}})=\tilde{f}(\widetilde{(a_n\circ b_n)_{n\geq 0}})=c\circ d=\tilde{f}(\widetilde{(a_n)_{n\geq 0}})\circ \tilde{f}(\widetilde{(b_n)_{n\geq 0}})$. So $\tilde{f}$ is a residuated lattice morphism.

For all $a\in A$, the constant sequence $(f(a))_{n\geq 0}\stackrel{\textstyle \rho _m}{\textstyle \rightarrow }f(a)\in C$, thus $\tilde{f}(\varphi _s(a))=\tilde{f}(\widetilde{(a)_{n\geq 0}})=f(a)$. So $\tilde{f}\circ \varphi _s=f$. Now let $(a_n)_{n\geq 0}\in {\cal C}_s(A)$. By the above, there exists $c\in C$ such that $f(a_n)\stackrel{\textstyle \rho _m}{\textstyle \rightarrow }c$, so $\tilde{f}(\widetilde{(a_n)_{n\geq 0}})=c$. As above, one can show that $(f(a_n))_{n\geq 0}\approx (c)_{n\geq 0}$, thus, by Lemma \ref{l6.19}, $\tilde{s}(\widetilde{(a_n)_{n\geq 0}})=\lim _{n\rightarrow \infty }s(a_n)=\lim _{n\rightarrow \infty }m(f(a_n))=\lim _{n\rightarrow \infty }m(c)=m(c)=(m\circ \tilde{f})(\widetilde{(a_n)_{n\geq 0}})$. Hence $m\circ \tilde{f}=\tilde{s}$.\end{proof}

\begin{openproblem}
Prove that the morphism $\tilde{f}:\tilde{A}_s\rightarrow C$ from Theorem \ref{t6.2}, (\ref{t6.2(8)}) is unique.
\end{openproblem}

Now let us analyse the construction of $\tilde{A}_s$ for the particular case when $L=([0,1],\max ,\min ,\odot _L,\rightarrow _L)$ is the standard MV-algebra, which is Cauchy-complete, as one can easily deduce from the fact that, if $d$ is the Euclidean distance in $\R $ restricted to $[0,1]\times [0,1]$, then $([0,1],d)$ is a complete metric space, and from the computation: for all $x,y\in L=[0,1]$, $d_L(x,y)=\min \{\min \{1,1-x+y\},\min \{1,1-y+x\}\}=\min \{1,1-x+y,1-y+x\}=\min \{1-x+y,1-y+x\}=1-\max \{x-y,y-x\}=1-|x-y|=1-d(x,y)$ (the deduction can be made in a similar manner to the one below that shows that $(A,\delta _s)$ is a complete pseudo-metric space iff $A$ is $\rho _s$-complete). We are still in the framework: $A$ is a residuated lattice and $s:A\rightarrow L$ is an order-preserving type I state, thus, in this case, $s:A\rightarrow [0,1]$ is a Bosbach state, as Remark \ref{r3.7} shows.

Let us define the function $\delta _s:A^2\rightarrow [0,1]$, for all $a,b\in A$, $\delta _s(a,b)=1-\rho _s(a,b)$. A function $\delta _t$ can be defined in this way for any Bosbach state $t$ on any residuated lattice.

\begin{remark}
$\delta _s$ is a pseudo-metric on $A$. Indeed, by Proposition \ref{p6.6}, for all $a,b,c\in A$, $\rho _s(a,b)\odot _L\rho _s(b,c)\leq \rho _s(a,c)$, thus $(1-\delta _s(a,b))\odot _L(1-\delta _s(b,c))\leq 1-\delta _s(a,c)$, that is $\max \{0,1-\delta _s(a,b)-\delta _s(b,c)\}\leq 1-\delta _s(a,c)$, therefore $1-\delta _s(a,b)-\delta _s(b,c)\leq 1-\delta _s(a,c)$, that is $\delta _s(a,c)\leq \delta _s(a,b)+\delta _s(b,c)$. Moreover, by Lemma \ref{lA}, (\ref{lA(1)}), it follows that $\delta _s$ is a metric on $A$ iff $s$ is faithful. This is valid for any residuated lattice $A$ and any Bosbach state $s$ on $A$.
\label{r6.22}
\end{remark}

The remark above shows that $(A,\delta _s)$ is a pseudo-metric space, thus we can construct its metric completion. In order to accomplish this, let us notice that: a sequence $(a_n)_{n\geq 0}$ in the pseudo-metric space $(A,\delta _s)$ converges towards $a\in A$ (in the pseudo-metric sense) iff $\lim _{n\rightarrow \infty }\delta _s(a_n,a)=0$, that is $\lim _{n\rightarrow \infty }\rho _s(a_n,a)=1$, that is $a_n\stackrel{\textstyle \rho _s}{\textstyle \rightarrow }a$. Also, a sequence $(a_n)_{n\geq 0}\subseteq A$ is Cauchy in the pseudo-metric space $(A,\delta _s)$ iff $\lim _{n,m\rightarrow \infty }\delta _s(a_n,a_m)=0$ iff $\lim _{n,m\rightarrow \infty }\rho _s(a_n,a_m)=1$ iff $(a_n)_{n\geq 0}$ is $\rho _s$-Cauchy. It follows that $(A,\delta _s)$ is a complete pseudo-metric space iff $A$ is $\rho _s$-complete, and this is valid for an arbitrary residuated lattice $A$ and an arbitrary Bosbach state $s:A\rightarrow [0,1]$. It also follows that, with the definition above, ${\cal C}_s(A)$ is equal to the set of the Cauchy sequences of the pseudo-metric space $(A,\delta _s)$ and the binary relation $\sim $ on ${\cal C}_s(A)$ satisfies: for all $\underline{a}=(a_n)_{n\geq 0},\underline{b}=(b_n)_{n\geq 0}\in {\cal C}_s(A)$, $\underline{a}\sim \underline{b}$ iff $\lim _{n\rightarrow \infty }\rho _s(a_n,b_n)=1$ iff $\lim _{n\rightarrow \infty }\delta _s(a_n,b_n)=0$.

With the notations above, let us define the function $\tilde{\delta _s}:\tilde{A}_s\times \tilde{A}_s\rightarrow L$, for all $\underline{a}=(a_n)_{n\geq 0},\underline{b}=(b_n)_{n\geq 0}\in {\cal C}_s(A)$, $\tilde{\delta _s}(\tilde{\underline{a}},\tilde{\underline{b}})=\lim _{n\rightarrow \infty }\delta _s(a_n,b_n)=1-\lim _{n\rightarrow \infty }\rho _s(a_n,b_n)=1-\tilde{\rho _s}(\tilde{\underline{a}},\tilde{\underline{b}})$; $\tilde{\delta _s}$ is well defined because $\tilde{\rho _s}$ is well defined. By Proposition \ref{p6.18}, $\tilde{\rho _s}$ is an $L$-similarity relation on $\tilde{A}_s$. Moreover, by Theorem \ref{t6.2}, (\ref{t6.2(6)}), $\tilde{\delta _s}=1-\tilde{\rho _s}=1-\rho _{\tilde{s}}=\delta _{\tilde{s}}$ with the notation above Remark \ref{r6.22}, hence, by Remark \ref{r6.22} and since $\tilde{s}$ is a faithful Bosbach state by Proposition \ref{p6.20}, it follows that $\tilde{\delta _s}=\delta _{\tilde{s}}$ is a metric on $\tilde{A}_s$. The usual construction from the theory of metric spaces identifies $(\tilde{A}_s,\tilde{\delta _s}=\delta _{\tilde{s}})$ to be the metric completion of $(A,\delta _s)$. The universality property of the metric completion ensures us that, for any Cauchy-complete metric space $C$ and any isometry $f:A\rightarrow C$, there exists a unique isometry $\tilde{f}:\tilde{A}_s\rightarrow C$ such that $\tilde{f}\circ \varphi _s=f$. We can translate this as the theorem below, by relying on the following lemma.

\begin{lemma}
Let $A_1$ and $A_2$ be two residuated lattices, $s_1:A_1\rightarrow [0,1]$ and $s_2:A_2\rightarrow [0,1]$ Bosbach states and $h:A_1\rightarrow A_2$ a morphism of residuated lattices. Then: $s_2\circ h=s_1$ iff $h$ is an isometry between the pseudo-metric spaces $(A_1,\delta _{s_1})$ and $(A_2,\delta _{s_2})$.
\end{lemma}
\begin{proof}
We shall use the fact that $h$ is a residuated lattice morphism and thus it preserves the biresiduum.

\noindent ``$\Rightarrow $``: Assume that $s_2\circ h=s_1$ and let $a,b\in A_1$. $\delta _{s_2}(h(a),h(b))=1-\rho _{s_2}(h(a),h(b))=1-s_2(d_{A_2}(h(a),h(b)))=1-s_2(h(d_{A_1}(a,b)))=1-s_1(d_{A_1}(a,b))=1-\rho _{s_1}(a,b)=\delta _{s_1}(a,b)$. Hence $h$ is an isometry.

\noindent ``$\Leftarrow $``: Assume that $h$ is an isometry, that is, for all $a,b\in A_1$, $\delta _{s_2}(h(a),h(b))=\delta _{s_1}(a,b)$. Let $a\in A_1$. By Lemma \ref{l2.2}, (\ref{l2.2(1)}) and (\ref{l2.2(2)}), $s_2(h(a))=s_2(h(d_{A_1}(a,1)))=s_2(d_{A_2}(h(a),h(1)))=\rho _{s_2}(h(a),h(1))=1-\delta _{s_2}(h(a),h(1))=1-\delta _{s_1}(a,1)=\rho _{s_1}(a,1)=s_1(d_{A_1}(a,1))=s_1(a)$. Hence $s_2\circ h=s_1$.\end{proof}

\begin{theorem}
For any residuated lattice $C$, any faithful Bosbach state $m:C\rightarrow [0,1]$ such that $(C,\delta _m)$ is a Cauchy-complete metric space, and any residuated lattice morphism $f:A\rightarrow C$ such that $m\circ f=s$, there exists a unique residuated lattice morphism $\tilde{f}:\tilde{A}_s\rightarrow C$ such that $m\circ \tilde{f}=\tilde{s}$ and $\tilde{f}\circ \varphi _s=f$.
\label{t01}
\end{theorem}

\begin{remark}
By an observation above, the fact that $(C,\delta _m)$ is Cauchy-complete is equivalent to the fact that $C$ is $\rho _m$-complete and hence the unique morphism $\tilde{f}$ from Theorem \ref{t01} is none other than the morphism constructed in the proof of Theorem \ref{t6.2}, (\ref{t6.2(8)}).
\end{remark}

\begin{proposition}
$\tilde{A}_s$ is $\rho _{\tilde{s}}$-complete, $\sigma $-complete and involutive, and $\tilde{s}$ is $\uparrow $-continuous in $1$. If $A$ is an MV-algebra, then $\tilde{s}$ is continuous.
\end{proposition}
\begin{proof}
By the above, $(\tilde{A}_s,\delta _{\tilde{s}})$ is a complete metric space, thus $\tilde{s}$ is a faithful Bosbach state and $\tilde{A}_s$ is a $\rho _{\tilde{s}}$-complete residuated lattice. Since $[0,1]$ with the natural order is $\sigma $-complete, it follows by Proposition \ref{p6.14} that $A$ is $\sigma $-complete and $s$ is $\uparrow $-continuous in $1$.

If $A$ is an MV-algebra, then obviously $\tilde{A}_s$ is an MV-algebra. $[0,1]$ is involutive, as any MV-algebra is. By Remark \ref{r6.15}, it follows that $\tilde{s}$ is continuous.\end{proof}

Adopting a denomination from \cite{ioanal3}, we shall call $\tilde{A}_s$ {\em the $s$-completion of $A$}.

\section{Final Remarks}
\label{finalremarks}

\hspace*{10pt} In this section we will sketch two ways in which we can relate generalized Bosbach states to monoidal t-norm-based logics and we will formulate some open problems.

\noindent (I) The probabilistic logic FP(\L $_n$,\L ) studied in \cite{flago}, \cite{fla1} is a formal description of a way of reasoning on the probability of fuzzy events through the infinite-valued \L ukasiewicz logic \L . In \cite{flago}, \cite{fla1} the authors admit the hypothesis that fuzzy events follow the rules of the finite-valued \L ukasiewicz logic \L $_n$.

We shall sketch now a more general context for developping some logics similar to FP(\L $_n$,\L ). Let ${\cal C}_1$ and ${\cal C}_2$ be two schematic extensions of the MTL logic (\cite{fla1}). The probabilistic logic ${\rm FP}({\cal C}_1,{\cal C}_2)$ is based on the following hypotheses:

\begin{itemize}
\item the events are structured by the logic ${\cal C}_1$;
\item the evaluation of the probability of the events is made in conformity to the logic ${\cal C}_2$.
\end{itemize}

The language of the logic ${\rm FP}({\cal C}_1,{\cal C}_2)$ is constructed by starting from a numerable set of propositional variables $V=\{p_1,p_2,\ldots ,p_k,\ldots \}$, the truth constant $\perp $, the connectives $\vee ,\wedge ,\rightarrow ,\& $ and a symbol $P$ (for the modality ``probably``). The formulas of ${\rm FP}({\cal C}_1,{\cal C}_2)$ are defined in two steps:

\begin{itemize}
\item the set $F_m(V)$ of the {\em non-modal} formulas is exactly the set of the formulas of ${\cal C}_1$ (the non-modal formulas will be denoted $\varphi ,\ \psi ,\ldots $);
\item the atomic modal formulas are of the form $P(\varphi )$, with $\varphi \in F_m(V)$; the set $MF_m(V)$ of the modal formulas is constructed inductively, starting from the atomic modal formulas and using the connectives $\vee ,\wedge ,\rightarrow ,\& $ and the truth constant $\perp $.
\end{itemize}

${\rm FP}({\cal C}_1,{\cal C}_2)$ has the following axioms:

\begin{itemize}
\item the axioms of ${\cal C}_1$ for non-modal formulas;
\item the axioms of ${\cal C}_2$ for modal formulas;
\item the following axioms for the modality $P$:

\noindent (A1) $P(\varphi \rightarrow \psi )\rightarrow (P(\varphi )\rightarrow P(\psi ))$

\noindent (A2) $P(\varphi \rightarrow \psi )\rightarrow (P(\varphi )\rightarrow P(\varphi \wedge \psi ))$
\end{itemize}

${\rm FP}({\cal C}_1,{\cal C}_2)$ has two deduction rules:

\begin{itemize}
\item the modus ponens rule (for modal and non-modal formulas);
\item the necessity rule: from $\varphi $ derive $P(\varphi )$.
\end{itemize}

Remark \ref{r3.26} shows that the logic FP(\L $_n$,\L ) can be obtained from ${\rm FP}({\cal C}_1,{\cal C}_2)$ by setting ${\cal C}_1=$\L $_n$ and ${\cal C}_2=$\L .

\begin{openproblem}
Define a semantics corresponding to the logic ${\rm FP}({\cal C}_1,{\cal C}_2)$ (by extending the notions of weak probabilistic Kripke model and strong probabilistic Kripke model from \cite{flago}, \cite{fla1}) and prove the weak and strong completeness theorems for ${\rm FP}({\cal C}_1,{\cal C}_2)$.
\label{o7.1}
\end{openproblem}

\noindent (II) Let ${\cal C}$ be a schematic extension of MTL and ${\cal C}_{\forall }$ be the predicate logic associated to ${\cal C}$ (see \cite{haj}, \cite{[a]}). We shall denote by $E$ the set of the sentences of ${\cal C}_{\forall }$ and by $E/_{\sim }=\{\hat{\varphi }|\varphi \in E\}$ the Lindenbaum-Tarski algebra of ${\cal C}_{\forall }$. $E/_{\sim }$ is an MTL-algebra that also verifies the algebraic form of the axioms specific to ${\cal C}_{\forall }$.

Let $D$ be a subset of $E$ such that:

\begin{itemize}
\item $D$ contains all the formal theorems of ${\cal C}_{\forall }$;
\item $D$ is closed with respect to the connectives $\vee ,\wedge ,\rightarrow ,\& $ and $D$ contains the truth constant $\perp $.
\end{itemize}

Then $D/_{\sim }=\{\hat{\varphi }|\varphi \in D\}$ is a subalgebra of $E/_{\sim }$ (in particular, $D/_{\sim }$ is an MTL-algebra). We consider on $[0,1]$ the structure of MTL-algebra induced by a left-continuous t-norm (\cite{beloh}).

\begin{definition}
A function $\mu :D\rightarrow [0,1]$ is called a {\em logical probability} on $D$ iff, for all $\varphi ,\psi \in D$:

\noindent (P1) if $\vdash \varphi $ then $\mu (\varphi )=1$;

\noindent (P2) $\mu (\varphi \rightarrow \psi )\rightarrow (\mu (\varphi )\rightarrow \mu (\psi ))=1$;

\noindent (P3) $\mu (\varphi \rightarrow \psi )=\mu (\varphi )\rightarrow \mu (\varphi \wedge \psi )$.
\label{d7.2}
\end{definition}

\begin{lemma}

Let $\mu :D\rightarrow [0,1]$ be a logical probability and $\varphi ,\psi \in D$. Then:

\begin{enumerate}
\item\label{l7.3(1)} if $\vdash \varphi \rightarrow \psi $ then $\mu (\varphi )\leq \mu (\psi )$;
\item\label{l7.3(2)} if $\vdash \varphi \leftrightarrow \psi $ then $\mu (\varphi )=\mu (\psi )$.
\end{enumerate}
\label{l7.3}
\end{lemma}

By Lemma \ref{l7.3}, (\ref{l7.3(2)}), we can define a function $\tilde{\mu }:D/_{\sim }\rightarrow [0,1]$ by $\tilde{\mu }(\hat{\varphi })=\mu (\varphi )$ for all $\varphi \in D$. It immediately follows that $\tilde{\mu }$ is an order-preserving type I state on the residuated lattice $D/_{\sim }$.

Let $U$ be a set of new constants and ${\cal C}_{\forall }(U)$ be the language obtained from ${\cal C}_{\forall }$ by adjoining the constants from $U$. We denote by $E(U)$ the set of the constants of ${\cal C}_{\forall }(U)$.

We fixe a set of constants $U$ and a logical probability $m:E(U)\rightarrow [0,1]$. We shall introduce two conditions on the pair $(U,m)$:

\noindent $(G\exists )$ for any formula $\phi (x)$ of ${\cal C}_{\forall }(U)$, $m(\exists x\phi (x))=\sup \{m(\bigvee _{i=1}^{n}\phi (a_i))|n\in \N ^{*},a_1,\ldots ,a_n\in U\}$;

\noindent $(G\forall )$ for any formula $\phi (x)$ of ${\cal C}_{\forall }(U)$, $m(\forall x\phi (x))=\inf \{m(\bigwedge _{i=1}^{n}\phi (a_i))|n\in \N ^{*},a_1,\ldots ,a_n\in U\}$.

$(G\exists )$ and $(G\forall )$ are similar to the Gaifman conditions on the probabilities defined in classical first-order logic (\cite{[b]}).

A {\em probabilistic structure} on ${\cal C}_{\forall }$ is a pair $(U,m)$ that satisfies $(G\exists )$ and $(G\forall )$. A probabilistic structure $(U,m)$ is a {\em probabilistic model of a logical probability $\mu :D\rightarrow [0,1]$} iff $m\mid _{D}=\mu $.

\begin{openproblem}
Prove for some schematic extensions ${\cal C}$ of the MTL logic the following completeness theorem: any logical probability $\mu $ admits a probabilistic model.
\label{o7.4}
\end{openproblem}

In the case of classical first-order logic, the enunciation above is Gaifman`s completeness theorem (\cite{[b]}). If ${\cal C}$ is the infinite-valued \L ukasiewicz logic \L , then such a completeness theorem is valid (\cite{[c]}).

\end{document}